%% file: neutralization.tex
\documentclass[reqno,   twoside, a4paper, 11pt]{amsart}
\usepackage[english]{babel}
\usepackage{amsmath}
\usepackage[latin1]{inputenc}
\usepackage{amsfonts}
\usepackage{amssymb}
\usepackage{amsthm}
\usepackage{fixme}
\usepackage{adjustbox}
\usepackage{MnSymbol}
\usepackage{stmaryrd}
\usepackage{microtype}
\usepackage{mathrsfs}
\usepackage{enumitem}
\usepackage{tikz}
\usepackage{tikz-cd}
\usetikzlibrary{matrix, arrows}
\tikzset{commutative diagrams/.cd, column sep/normal=3ex}
\newcount\hh
\newcount\mm
\mm=\time
\hh=\time
\divide\hh by 60
\divide\mm by 60
\multiply\mm by 60
\mm=-\mm
\advance\mm by \time
\def\hhmm{\number\hh:\ifnum\mm<10{}0\fi\number\mm}

\renewcommand{\phi}{\varphi}
\DeclareMathOperator{\image}{im}
\DeclareMathOperator{\rs}{rs}

\DeclareMathOperator{\Rlim}{R^1\varprojlim}

\DeclareMathOperator{\supp}{supp}

\DeclareMathOperator{\DM}{DM}

\DeclareMathOperator{\sep}{sep}

\DeclareMathOperator{\rank}{rank}

\renewcommand{\int}{\text{int}}

\DeclareMathOperator{\spc}{special}
\renewcommand{\triangle}{\spc}
\DeclareMathOperator{\unip}{unip}

\DeclareMathOperator{\Spec}{Spec}

\DeclareMathOperator{\ord}{ord}
\DeclareMathOperator{\cont}{cont}

\DeclareMathOperator{\Hom}{Hom}

\DeclareMathOperator{\Isom}{Isom}

\DeclareMathOperator{\alg}{alg}
\DeclareMathOperator{\Conn}{Conn}
\DeclareMathOperator{\Repf}{Repf}

\DeclareMathOperator{\triv}{triv}

\DeclareMathOperator{\Gal}{Gal}
\DeclareMathOperator{\id}{id}

\DeclareMathOperator{\GL}{GL}

\DeclareMathOperator{\rsn}{ 0}
\DeclareMathOperator{\rsi}{ \infty}

\DeclareMathOperator{\Frac}{Frac}
\DeclareMathOperator{\Cov}{Cov}
\DeclareMathOperator{\Ext}{Ext}
\DeclareMathOperator{\Vectf}{Vectf}

\DeclareMathOperator{\Strat}{Strat}

\DeclareMathOperator{\restr}{res}
\DeclareMathOperator{\et}{\text{\'et}}

\newcommand{\Z}{\mathbb{Z}}
\newcommand{\G}{\mathbb{G}}

\newcommand{\N}{\mathbb{N}}
\newcommand{\C}{\mathbb{C}}
\newcommand{\A}{\mathbb{A}}
\newcommand{\F}{\mathbb{F}}
\newcommand{\llparen}{(\!(}
\newcommand{\rrparen}{)\!)}
\renewcommand{\P}{\mathbb{P}}

\renewcommand{\subset}{\subseteq}

\theoremstyle{theorem}
\newtheorem{theorem}{Theorem}%[section]

\newtheorem{proposition}{Proposition}[section]
\newtheorem{lemma}[proposition]{Lemma}

\newtheorem{corollary}[proposition]{Corollary}

\theoremstyle{definition} %definition
\newtheorem{definition}[proposition]{Definition}
\newtheorem{convention}[proposition]{Convention}

\newtheorem{remark}[proposition]{Remark}

\setenumerate{label=\emph{\alph*)}, ref=\alph*)}

\title{Local-to-global extensions of $\mathscr{D}$-modules in positive
characteristic}
\author{Lars Kindler}
\address{Lars Kindler, Freie Universit\"at Berlin, Mathematisches Institut,
Arnimallee 3, 14195 Berlin, Germany}
\email{kindler@math.fu-berlin.de}
\thanks{The author was supported by ERC Advanced Grant 226257.}
\begin{document}
\begin{abstract}
	In \cite{Katz/Calculation}, Katz defines the notion of a
	\emph{special} flat connection on $\P^1_{\C}\setminus \{0,\infty\}$, and he
	shows that the functor which restricts a flat connection to
	the punctured disc around the point at infinity gives rise to an equivalence
	between the category of special flat connections on
	$\P^1_{\C}\setminus \{0,\infty\}$ and the category of
	differential modules on $\C\llparen t\rrparen$.

	In this article, we prove the corresponding statement over an
	algebraically closed field $k$ of positive
	characteristic.  The role of flat connections is played by vector bundles carrying an action of the
	(full) ring of differential operators.  Such objects are also called \emph{stratified
	bundles}. The formal local variant on the field of Laurent series
	$k\llparen t\rrparen$  is called
	\emph{iterated differential module}. We define the notion of a special
	stratified bundle on $\P^1_{k}\setminus \{0,\infty\}$, and show that
	restriction to the punctured disc around the point at infinity induces an equivalence between
	the category of special stratified bundles and iterated differential
	modules on $k\llparen t\rrparen$. This extends one of the main results of
	\cite{Katz/LocalToGlobal}, and has several interesting consequences
	which extend well-known statements about \'etale fundamental groups to 
	higher dimensional fundamental groups.
\end{abstract}
\maketitle
\input{maintext.tex}

\input{neutralization.bbl}

\end{document}

%% file: maintext.tex
Let $\G_{m,\C}:=\P^1_{\C}\setminus\{0,\infty\}$, and let $\Conn(\G_{m,\C})$
be the
category of vector bundles on $\G_{m,\C}$ equipped with a flat connection.
Consider $\Spec\C\llparen t\rrparen$ as the punctured disc around $\infty$, and write $\DM(\C\llparen t\rrparen)$ for the category of
differential modules on $\C\llparen t\rrparen$. Restriction induces a functor $\Conn(\G_{m,\C})\rightarrow
\DM(\C\llparen t\rrparen)$. In \cite{Katz/Calculation}, Katz defines a
full subcategory \[\Conn^{\spc}(\G_{m,\C})\subset \Conn(\G_{m,\C})\] of
\emph{special} flat connections, with the property that the restriction
functor
\begin{equation}\label{eq:complexEq}\Conn^{\spc}(\G_{m,\C})\rightarrow
	\DM(\C\llparen t\rrparen)\end{equation}
is an equivalence. It follows that the category of
differential modules on $\C\llparen t\rrparen$ can be equipped with a
$\C$-valued fiber functor. The main result of this article is a generalization of
Katz' equivalence \eqref{eq:complexEq} to positive characteristic.

Fix an algebraically closed field $k$ of positive characteristic $p$. If $X$ is a smooth $k$-variety, then we write $\Strat(X)$ for the category of
stratified bundles on $X$, i.e.~for the category of $\mathcal{O}_X$-coherent
$\mathscr{D}_{X/k}$-modules, where $\mathscr{D}_{X/k}$ is the sheaf 
of differential operators on $X$ relative to $k$, as defined in
\cite[\S16.]{EGA4}\footnote{The notion of a stratification goes back to
	\cite{Grothendieck/Crystals}. On a smooth $k$-variety $X$, the datum of a
	stratification on an $\mathcal{O}_X$-module  $E$ is equivalent to
	an action of $\mathscr{D}_{X/k}$ on $E$
	(\cite[Ch.~2]{BerthelotOgus/Crystalline}) and if $E$ is
	$\mathcal{O}_X$-coherent, then
	it is locally free (\cite[p.314]{Saavedra}).}. In many respects, the behavior of these objects
is very similar to the behavior of vector bundles with flat connections on smooth
complex varieties, see e.g.~\cite{Gieseker/FlatBundles}, \cite{EsnaultMehta/Gieseker}, \cite{Esnault/ECM},
\cite{Kindler/FiniteBundles}. In particular, after chosing a base point,
$\Strat(X)$ is a neutral tannakian category over $k$, see \cite[Ch.~VI, \S.1]{Saavedra}.

An \emph{iterated differential module} on $k\llparen t\rrparen$ is a finite dimensional $k\llparen t\rrparen$-vector
space together with an action of the (commutative) $k$-algebra
$k[\delta_{t}^{(n)}|n\geq 0]=k[\delta^{(p^n)}_t|n\geq 0]$ (see
\cite[Sec.~6]{Matzat}). Here $\delta^{(n)}_t$ is the differential
operator which acts on powers of $t$ via the formula
\begin{equation}\label{eq:diffOp}\delta^{(n)}_t(t^r)=\binom{r}{n}t^r.\end{equation}
Abusing
notation slightly,  we will write $\Strat(k\llparen
t\rrparen)$ for the category of iterated differential modules on
$k\llparen t\rrparen$, and also say ``stratified bundle on $k\llparen
t\rrparen$'' (even though $k\llparen t\rrparen$ is not a $k$-algebra of finite
type). The category
$\Strat(k\llparen t\rrparen)$ is a $k$-linear, abelian, rigid tensor category with unit
$\mathbf{1}:=k\llparen t\rrparen$ with its natural $k[\delta^{(n)}_t|n\geq
0]$-action given by \eqref{eq:diffOp}. 
One has two restriction functors $\Strat(\G_m)\rightarrow
\Strat(k\llparen t\rrparen)$, and $\Strat(\G_m)\rightarrow \Strat(k\llparen
t^{-1}\rrparen)$ arising from the inclusions $k[t^{\pm 1}]\subset
k\llparen t\rrparen$, $t\mapsto t$ and $k[t^{\pm 1}]\subset k\llparen t^{-1}\rrparen$,
$t\mapsto t$.
\begin{convention}\label{conv:inclusion}For notational ease, the roles of $0$ and $\infty$
	in the following will be opposite to their roles in
	\cite{Katz/Calculation}, \cite{Katz/LocalToGlobal}.
	In other words, will always consider the polynomial ring $k[t]$ as a subring of $k[t^{\pm 1}]$ via the
	 inclusion defined by mapping $t\mapsto t$. Similarly, we will consider
	$k[t^{\pm 1}]$ as a subring of both $k\llparen t\rrparen$ and $k\llparen
	t^{-1}\rrparen$ via the  inclusions given by $t\mapsto t$.  From now
	on, we also simply write $\G_m:=\G_{m,k}$.
\end{convention}

We can now formulate the first main result of this article.
\begin{theorem}\label{thm:equivalence}
	There
	is a full subcategory $\Strat^{\spc,\infty}(\G_m)\subset \Strat(\G_m)$,
	such that the restriction functor
	\[\restr:\Strat(\G_m)\rightarrow \Strat(k\llparen t\rrparen), E\mapsto
	E|_{k\llparen t\rrparen}:=E\otimes_{k[t^{\pm 1}]}k\llparen t\rrparen,\]
	given by the inclusion $k[t^{\pm 1}]\subset k\llparen t\rrparen$, induces an equivalence
	\[\Strat^{\triangle,\rsi}(\G_m)\xrightarrow{\cong}\Strat(k\llparen t
		\rrparen).\]
	Moreover the notion of a special stratified bundle on $\G_m$ extends
	the notion of a special covering of \cite{Katz/LocalToGlobal}, see
	Proposition \ref{prop:special}.
\end{theorem}
Just as in \cite[Cor.~2.4.12]{Katz/Calculation}, we obtain from Theorem
\ref{thm:equivalence} the following consequence.
\begin{corollary}\label{cor:neutralization}
	There exists an exact $\otimes$-functor
	$\omega:\Strat(k\llparen t\rrparen)\rightarrow \Vectf_k$, where
	$\Vectf_k$ denotes the category of finite dimensional $k$-vector
	spaces.
\end{corollary}
As the last part of the theorem suggests, Theorem \ref{thm:equivalence} is closely related to the main theorem of
\cite[Part 1]{Katz/LocalToGlobal}.
Exchanging the roles of $0$ and $\infty$ in Theorem \ref{thm:equivalence},
we obtain a full subcategory $\Strat^{\spc,\rsn}(\G_m)$, such that restriction
induces an equivalence
\begin{equation}\label{eq:oppositeEq}\Strat^{\spc,\rsn}(\G_m)\xrightarrow{\cong}\Strat(k\llparen
	t^{-1}\rrparen).\end{equation}
Let $\Cov(\G_m)$ denote the
category of finite \'etale coverings of $\G_m$. 
 In \cite{Katz/LocalToGlobal}, Katz and Gabber define a subcategory
$\Cov^{\spc}(\G_m)$ of special coverings of $\G_m$, such that the restriction
functor  $\Cov(\G_m)\rightarrow \Cov(k\llparen t^{-1}\rrparen)$ induces an
equivalence
\begin{equation}\label{eq:finiteEq}\Cov^{\spc}(\G_m)\xrightarrow{\cong}
	\Cov(k\llparen t^{-1}\rrparen).
\end{equation}
For an \'etale covering $f:X\rightarrow \G_m$ the
$\mathcal{O}_X$-module $f_*\mathcal{O}_X$ is naturally a stratified bundle on
$\G_m$; we obtain a faithful functor
$\Cov(\G_m)\rightarrow \Strat(\G_m)$. Theorem
\ref{thm:equivalence} claims that $f$ is special in the sense of
\cite{Katz/LocalToGlobal} 
if and only if $f_*\mathcal{O}_{X}\in
\Strat^{\spc,\rsn}(\G_m)$. This means that the equivalence
\eqref{eq:oppositeEq}
restricts to \eqref{eq:finiteEq}. \\

The structure of this article is as follows. In Section \ref{sec:Rlim} we
establish some facts about $\Rlim$ of a projective system of non-abelian
groups. These will be the main tools in the proof of Theorem
\ref{thm:equivalence}, as certain sets of isomorphism classes related to
stratified bundles can be naturally described as $\Rlim$ of such projective
systems. We explain these identifications in Section \ref{sec:stratAndRlim}. In Section
\ref{sec:specialStrat} we define the notion of a special stratified bundle.
The proof of Theorem \ref{thm:equivalence} is given in Section
\ref{sec:proof}. Finally, in Section \ref{sec:applications}, we present
applications of Theorem \ref{thm:equivalence} and the methods used in its
proof; we conclude this introduction with a brief summary of these
applications.

\begin{definition}If $X$ is a smooth, finite type $k$-scheme, and
	$\omega:\Strat(X)\rightarrow \Vectf_k$ a fiber functor, then we write
	$\pi_1^{\Strat}(X,\omega)$ for the group scheme attached to $\omega$ via
	the Tannaka formalism; it is a
	reduced (\cite{DosSantos}), affine $k$-group scheme. If $X=\Spec R$ is affine, we also
	write $\pi_1^{\Strat}(R,\omega)$ instead of $\pi_1^{\Strat}(\Spec
	R,\omega)$. If $R=k\llparen
	t\rrparen$, we denote by $\pi_1^{\Strat}(k\llparen t\rrparen,\omega)$ the affine
	$k$-group scheme attached to the category $\Strat(k\llparen
	t\rrparen)$, even though $k\llparen t\rrparen$
	is not a finite type $k$-algebra.

	In the same vein we write $\pi_1^{\unip}(X,\omega)$
	(resp.~$\pi_1^{\rs}(X,\omega)$) for the affine group schemes
	associated with the full subcategory of $\Strat(X)$ with objects the
	unipotent (resp.~regular singular) stratified bundles. 
\end{definition}

The first application is a straightforward consequence of Corollary
\ref{cor:neutralization}.

\begin{theorem}\label{thm:ses}
	Fix a fiber functor $\omega:\Strat(k\llparen t\rrparen)\rightarrow
	\Vectf_k$. There is a short exact sequence of reduced affine $k$-group schemes
	\begin{equation}\label{eq:ses}\begin{tikzcd}1\rar&P(\omega)\rar&\pi_1^{\Strat}(k\llparen t\rrparen,\omega)\rar&
		\pi_1^{\rs}(k\llparen t\rrparen,\omega)\rar&1.
	\end{tikzcd}\end{equation}
	The group scheme $P(\omega)$ is unipotent.
\end{theorem}

Theorem \ref{thm:ses} should be seen as a generalization of the usual short
exact sequence of profinite groups
\begin{equation}\label{eq:profinite-ses}\begin{tikzcd}
		1\rar&P\rar&\Gal(k\llparen t\rrparen^{\sep}/k\llparen t\rrparen)\rar& \widehat{\Z}^{(p')}\rar&1.
	\end{tikzcd}
\end{equation}
with the wild ramification group $P$ the unique pro-$p$-Sylow subgroup of
$\Gal(k\llparen t\rrparen^{\sep}/k\llparen t\rrparen)$.
In fact, \eqref{eq:profinite-ses} can be obtained up to inner automorphism
from \eqref{eq:ses} via profinite completion.

\begin{theorem}\label{thm:A1}
	Let $x\in \G_m$ be a closed point, and denote by
	$\omega_x: \Strat^{\unip}(\A^1)\rightarrow \Vectf_k$ the associated
	fiber functor. By Theorem \ref{thm:equivalence}, this induces a neutral
	fiber functor for
	$\Strat^{\unip}(k\llparen t \rrparen)$, which we also denote by
	$\omega_x$.
	The inclusion
	$k[t]\subset k\llparen t^{-1}\rrparen$
	induces an isomorphism
	\[\pi_1^{\unip}(k\llparen t^{-1}\rrparen,\omega_x)\xrightarrow{\cong}
		\pi_1^{\unip}(\A^1_k,\omega_x).\]
	       % 
	       % Recall that $\pi(\Strat^{\unip}(k\llparen
	       % t^{-1}\rrparen),\omega)$ is
	       % canonically isomorphic to the maximal pro-unipotent quotient of
	       % $\pi(\Strat(k\llparen t^{-1}\rrparen),\omega)$, and
	       % similarly for
	       % $\pi(\Strat^{\unip}(\A^1),\omega(-|_{k\llparen
	       % 	t^{-1}\rrparen}))$.
\end{theorem}
There are statements analogous to Theorem \ref{thm:A1} in various contexts:
In particular, see \cite[Prop.~1.4.2]{Katz/LocalToGlobal} for the case of the maximal pro-$p$
quotient of the \'etale fundamental group, and
\cite{Unver/PurelyIrregular} for the case of ``purely irregular'' connections over the complex
numbers.

We also obtain a coproduct formula analogous to \cite[p.~98]{Katz/LocalToGlobal}
and \cite[Thm.~2.10]{Unver/PurelyIrregular}:
\begin{theorem}\label{thm:coproduct}
	Let $x\in \G_m$ be a closed point and
	$\omega_x:\Strat^{\unip}(\G_m)\rightarrow \Vectf_k$ the associated
	fiber functor. By Theorem \ref{thm:equivalence}, $\omega_x$ induces
	fiber functors on $\Strat^{\unip}(k\llparen t\rrparen)$ and
	$\Strat^{\unip}(k\llparen t^{-1}\rrparen)$, which we also denote by
	$\omega_x$. The inclusions
	$k[t^{\pm 1}]\subset k\llparen t\rrparen$ and $k[t^{\pm
	1}]\subset k\llparen t^{-1}\rrparen$ induce an isomorphism
	\[ \pi_1^{\unip}(k\llparen
		t\rrparen,\omega_x)\ast^{\unip}\pi_1^{\unip}(k\llparen
		t^{-1}\rrparen,\omega_x)\xrightarrow{\cong}\pi_1^{\unip}(\G_m,\omega_x)
		\]
		where $\ast^{\unip}$ denotes the coproduct in the category of
	unipotent affine $k$-group schemes, see Definition
	\ref{defn:coproduct}.
\end{theorem}

\vspace{.3cm}
\emph{Acknowledgments:} The author thanks H.~Esnault and K.~R\"ulling for
helpful discussions and the referee for a careful reading of the manuscript.

\section{Derived inverse limits of nonabelian groups}\label{sec:Rlim}
We first recall the some notions from ``homological algebra'' for
noncommutative groups.

\begin{definition}[{\cite[IX.2]{Bousfield}}]
	If $G_1\xleftarrow{f_1} G_2\xleftarrow{f_2} \ldots$ is a projective system of
	groups, then we define the pointed set $\Rlim_n G_n$ as
	follows: The product $\prod_{n\in\N} G_n$ acts on itself from the left
	via
	\[(g_1,\ldots, g_n,\ldots)\cdot (x_1,\ldots,
		x_n,\ldots):=(g_1x_1f_1(g_2)^{-1},g_2x_2f_2(g_3)^{-1},\ldots),\]
	and $\Rlim_n G_n$ is the set of orbits of this action, pointed by the
	orbit containing $1\in \prod_n G_n$.
	This reduces to the usual definition of $\Rlim$ if the
	$G_n$ are abelian.
\end{definition}

\begin{remark}\label{remark:naturalnumbers}
	The general results about derived inverse limits of projective systems
	of arbitrary groups will follow from the following construction, due
	to Ogus (\cite{Hartshorne/derham}): Equip $\mathbb{N}$ with the
	topology in which the open sets are $\emptyset, \mathbb{N}$ and
	$[1,n], n\geq 1$. Then the category of projective systems of groups
	is equivalent to the category of sheaves of groups on this topological
	space. We can hence apply the theory of \cite[Ch.~3]{Giraud}. In fact,
	if we write $G$ for the sheaf on $\N$ associated with the
	projective system $\{G_n\}$,  then
	it is not difficult to functorially identify the pointed set $\Rlim_n G_n$ with
	the pointed set $H^1(\N, G)$, which by definition is the pointed set
	of isomorphism classes of $G$-torsors in the category of sheaves on
	$\N$.
\end{remark}

\begin{lemma}\label{lemma:Rlim}
We list some properties of the functor $\Rlim$:
\begin{enumerate}
	\item\label{item:Rlim1} If $G_1\xleftarrow{f_1} G_2\xleftarrow{f_2} \ldots$ is a projective
	system of groups with surjective transition morphisms, then $\Rlim_n
	G_n=0$.
\item\label{item:Rlim2} If
	\[
		\begin{tikzcd}
			1\rar& \{A_n\}\rar& \{B_n\}\rar& \{C_n\}\rar& 1
	\end{tikzcd}\]
	is a short exact sequence of projective systems of groups, then
	there is a functorial exact sequence of pointed sets
	\[
		\begin{tikzcd}[column sep=.26cm]
			1\rar& \varprojlim A_n\rar& \varprojlim B_n \rar& \varprojlim C_n\rar& \Rlim A_n\rar& \Rlim B_n\rar& \Rlim C_n\rar& 1.
		\end{tikzcd}
		\]
	Moreover, the map $\varprojlim C_n\rightarrow \Rlim A_n$ extends
	to a functorial action of $\varprojlim C_n$ on $\Rlim A_n$, which
	induces an injection
	\[\begin{tikzcd}
			\left(\varprojlim_n C_n\right)\backslash\Rlim A_n\rar[hook]& \Rlim B_n.
		\end{tikzcd}
		\]
\item\label{item:Rlim3} If in \ref{item:Rlim2} the groups $A_n$ are
	\emph{central} subgroups of $B_n$, then the map
	\[\Rlim A_n\rightarrow \Rlim B_n\] extends to a functorial action of the
	\emph{abelian group} $\Rlim
A_n$ on the pointed set $\Rlim B_n$, which  induces a bijection
\[\left(\Rlim A_n\right) \backslash \Rlim B_n\rightarrow \Rlim C_n.\]
%\item\label{item:Rlim3prime}If in \ref{item:Rlim3} the map $\{B_n\}\rightarrow
%	\{C_n\}$ admits a section $s=(s_n:C_n\rightarrow B_n)$, then the action of
%	$\Rlim A_n$ on $\Rlim B_n$ is free.\\
\end{enumerate}
%We will also need a slightly more general version of \ref{item:Rlim2}:\\
%\begin{enumerate}[resume]
%\item\label{item:Rlim4} If $\{A_n\}\hookrightarrow \{B_n\}$ is an injective morphism of
%	projective systems of groups (with $A_n$ not necessarily normal in
%	$B_n$), then there exists an exact sequence of
%	pointed sets
%	\[
%		\begin{tikzcd}
%			1\rar& \varprojlim A_n\rar& \varprojlim B_n\rar&
%			\varprojlim (A_n\backslash B_n)\rar& \Rlim A_n\rar& \Rlim
%		B_n,
%	\end{tikzcd}\]
%	and the map $\varprojlim B_n\rightarrow \varprojlim (A_n\backslash
%	B_n)$ extends
%	to an action of $\varprojlim B_n$ on $\varprojlim (A_n\backslash B_n)$,
%	inducing an injection
%	\[(\varprojlim B_n)\backslash  \varprojlim\left( A_n\backslash
%		B_n\right).\]
%	Moreover, if the projective system of pointed sets $\{B_n\backslash A_n\}$ satisfies
%	the Mittag-Leffler condition, then $\Rlim A_n\rightarrow \Rlim B_n$
%	is surjective.
%	\end{enumerate}
\end{lemma}
\begin{proof}
	From the definition of $\Rlim_n G_n$ one immediately verifies
	\ref{item:Rlim1}.
	For \ref{item:Rlim2} one either translates \cite[III.3.2]{Giraud} using
	Remark \ref{remark:naturalnumbers}, or one looks in
	\cite[IX.2]{Bousfield}. Statement \ref{item:Rlim3}  
	%and \ref{item:Rlim3prime} 
	follows from \cite[Prop.~III.3.4.5]{Giraud}. For
	the reader's
	convenience
       % , and to esablish the method of ``twisting'' a
       % projective system, 
	we check it by hand: Of course $\prod_n A_n$ acts on
	$\prod_n B_n$, and the orbits of this action are precisely the fibers
	of $\prod_n B_n \twoheadrightarrow \prod_n C_n$.
	The centrality of $A_n$ in $B_n$ guarantees that
	this action descends to the quotient map $\Rlim_n A_n\rightarrow
	\Rlim_n B_n$. This also shows that the orbits of the action are
	precisely the fibers of $\Rlim_n B_n\rightarrow \Rlim_n C_n$, which
	completes the proof of \ref{item:Rlim3}.
\end{proof}

\begin{definition}\label{defn:strongexactness}For notational convenience, we
	say that a short exact sequence of
	pointed sets
	%\[0\rightarrow A\rightarrow B \rightarrow C\rightarrow 0\]
	\[\begin{tikzcd}
			1\rar &A\rar& B \rar& C\rar& 1
		\end{tikzcd}\]
	is \emph{strongly exact}\footnote{This is nonstandard terminology.}
	if the following hold:
	\begin{itemize}
		\item $A$ is an abelian group.
		\item $A\rightarrow B$ is injective and extends to an
			action
			of $A$ on $B$, with orbits precisely the fibers of
			$B\rightarrow C$, i.e.~such that one gets an induced
			bijection $A\backslash B\rightarrow C$.
	\end{itemize}
	A \emph{morphism of strongly exact short exact sequences} is a
	morphism of short exact sequences of pointed sets
       % \[\xymatrix{
       % 	0\ar[r]&A\ar[d]^f\ar[r]&B\ar[r]\ar[d]^g&C\ar[r]\ar[d]^h& 0\\
       % 	0\ar[r]&A'\ar[r]&B'\ar[r]&C'\ar[r]& 0
       % }\]
	\[
		\begin{tikzcd}
			1\rar&A\dar{f}\rar&B\rar\dar{g}&C\rar\dar{h}& 1\\
		1\rar&A'\rar&B'\rar&C'\rar& 1
		\end{tikzcd}
		\]
	such that $f$ is a homomorphism of groups and such that $g$ is
	compatible with the $A$-action on $B$ and the $A'$-action on $B'$, i.e.~$g(ab)=f(a)g(b)$ for
	$a\in A, b\in B$.
\end{definition}

We will need a weak version of the Snake lemma for strongly exact short exact
sequences of pointed sets:

\begin{lemma}\label{lemma:snake}
	Let
%	\[\xymatrix{
%		0\ar[r]&G\ar[d]^{f}\ar[r]^i& M \ar[r]^j
%		\ar[d]^{g}&N\ar[d]^{h} \ar[r] &0\\
%		0\ar[r]&G'\ar[r]^{i'}& M' \ar[r]^{j'} &N' \ar[r] &0}
		%\]
	\[
	\begin{tikzcd}
		1\rar&G\dar{f}\rar{i}& M \rar{j} \dar{g}&N\dar{h} \rar &1\\
		1\rar&G'\rar{i'}& M' \rar{j'} &N' \rar &1
	\end{tikzcd}
	\]
	be a morphism of \emph{strongly exact} (Definition
	\ref{defn:strongexactness}) short exact sequences of pointed
	sets.
	\begin{enumerate}[label=\emph{\alph*)},ref={\alph*)}]
		\item\label{item:snake1} If $f$ is surjective, then
			$i$ and $j$ induce a strongly exact short
			exact sequence of pointed sets
		%	\begin{equation}%\label{kernelsequence}
		%		0\rightarrow
		%		\ker(f)\xrightarrow{i|_{\ker(f)}}
		%		\ker(g)\xrightarrow{j|_{\ker(g)}}
		%		\ker(h)\rightarrow 0.\end{equation}
			\begin{equation}\label{kernelsequence}
				\begin{tikzcd}[column sep=1cm]
					1\rar&
				\ker(f)\rar{i|_{\ker(f)}}&
				\ker(g)\rar{j|_{\ker(g)}}&
				\ker(h)\rar& 1.
			\end{tikzcd}
		\end{equation}
		\item\label{item:snake2} If $f$ and $h$ are surjective, then
			$g$ is surjective.
		\end{enumerate}
\end{lemma}
\begin{proof}
	For \ref{item:snake1}, it is clear that the sequence
	\eqref{kernelsequence} is exact on the left and in the middle.
	We check that
	$\ker(g)\rightarrow \ker(h)$ is surjective. Let $x\in \ker(h)$ and let $y\in M$
	be a lift of $x$. Then $j'g(y)=1$, so $g(y)\in G'$. Let $y'\in G$ be a
	lift of $g(y)$. Then $y'^{-1}\cdot y\in \ker(g)$ is a lift of $x$.

	Moreover, $\ker(f)$ still is an abelian group, and the action of $G$
	on $M$ induces an action of $\ker(f)$ on $M$. It is straightforward to
	check that this action stabilizes $\ker(g)$. If $m,m'\in \ker(g)$ and
	$x\in \ker(f)$, such that $m=xm'$, then $j(m)=j(xm')$. Conversely, if 
	$j(m)=j(m')$, then there exists $x\in G$ such that $m=xm'$. But
	this means $1=g(m)=g(xm')=f(x)g(m')=f(x)$, so $x\in \ker(f)$, and $m$
	and $m'$ are in the same $\ker(f)$-orbit. This means that the fibers
	of $j|_{\ker(g)}$ are precisely the orbits of the $\ker(f)$-action on
	$\ker(g)$. This finishes the proof of \ref{item:snake1}.

	For \ref{item:snake2}, 
       % 
       % one can do the usual diagram chase using that the actions of $G$ and
       % $G'$ are free. Indeed, assume that $x,y\in M$ are elements such that
       % $g(x)=g(y)$. Then, as $h$ is bijective, we see that $x$ and $y$ map to
       % the same element of $N$. This means that $x$ and $y$ lie in the same
       % $G$-orbit of $M$, say $y=ax$ for $a\in G$. But then
       % $g(x)=g(y)=f(a)g(x)$. Since $f$ is injective, and since the action of
       % $G'$ is free, it follows that $a=1$, so $x=y$, and $g$ is injective.
	%To see that $g$ is surjective, 
	let $m'\in M'$. By the surjectivity of $h$ and $j$, there exists $m\in M$, such that
	$hj(m)=j'(m')$. Then $g(m)$ and $m'$ lie in the same $G'$-orbit of
	$M'$, say $m'=x'g(m)$ for some $x'\in G'$. There exists $x\in G$  with
	$f(x)=x'$, and $g(xm)=f(x)g(m)=x'g(m)=m'$, so $g$ is surjective.
\end{proof}
\section{Stratified bundles and $\Rlim$}\label{sec:stratAndRlim}
We continue to denote by $k$ an algebraically closed field of positive
characteristic $p$. 
\begin{definition}
	For an integer $r> 0$, we write $B_r\subset \GL_r$ for the smooth affine $k$-group scheme, such that
	for a $k$-algebra $R$ the group $B_r(R)$ is the group of
	invertible upper
	triangular $r\times r$-matrices, and $U_r\subset B_r$ for the smooth $k$-subgroup scheme
	given by the unipotent upper triangular matrices. Note that we have a split short exact sequence of group
	schemes
	%\[1\rightarrow U_r \rightarrow B_r\rightarrow \G_m^{r}\rightarrow 1\]
	\[\begin{tikzcd}1\rar& U_r \rar& B_r\rar& \G_m^{r}\rar\lar[swap, bend
		right=45]& 1\end{tikzcd}\]
	where the map $B_r\rightarrow \G_m^r$ is defined by sending an upper triangular matrix to the associated diagonal
	matrix.

	If $R$ is a $k$-algebra we write $R^{p^n}$ for the subring of
	$p^n$-th powers in $R$, i.e.~for the image of the map
	$R\xrightarrow{(\cdot)^{p^n}}R$. If is $G$ a group scheme, we write $\{G(R^{p^n})\}_{n}$ for the
	projective system with transition maps the homomorphisms
	$G(R^{p^{n+1}})\rightarrow G(R^{p^n})$ induced by the inclusion
	$R^{p^{n+1}}\hookrightarrow R^{p^n}$. To simplify notation,
	we write $\Rlim_n G(R^{p^n}):=\Rlim_n \{G(R^{p^n})\}_n$.
\end{definition}
\begin{proposition}\label{prop:stratAndRlim}Let $R=k[t^{\pm 1}]$, $k\llparen
	t\rrparen$ or $k\llparen t^{-1}\rrparen$.
	\begin{enumerate}
		\item The pointed set of isomorphism classes of stratified
			bundles of rank $r$ on $R$ can be identified with
			\[\Rlim_n \GL_r(R^{p^n}).\]
		\item The pointed set of isomorphism classes of composition
			series of stratified bundles
			\begin{equation}\label{eq:compSeries}0\subsetneqq E_1\subsetneqq
				E_2\subsetneqq\ldots\subsetneqq
				E_r\end{equation}
				with $\rank E_i/E_{i-1}=1$ (resp.~with
				$E_i/E_{i-1}\cong R$ as stratified bundles) can be identified
				with
				\[\Rlim_n
					B_r(R^{p^n})\quad\left(\text{resp.~}\Rlim_nU_r(R^{p^n})\right).\]
			\item The functorial long exact sequence obtained by
				applying $\varprojlim$ to the split short exact sequence of
				projective systems
				\[\begin{tikzcd}
						1\rar&\{U_r(R^{p^n})\}_n\rar&
						\{B_r(R^{p^n})\}_n\rar&\{\G_m^r(R^{p^n})\}_n\rar&
						1
					\end{tikzcd}
					\]
					induces a short exact sequence of pointed sets 
%			\begin{equation}\label{Trsequence} \xymatrix@C=.5cm{
%					1\ar[r]& \Rlim_n U_r(R^{p^n})\ar[r]& \Rlim_n
%					B_r(R^{p^n})\ar[r]& \Rlim_n
%					\G_m^{r}(R^{p^n})\ar[r]\ar@/_.5cm/[l]_{s_R}&
%				1}\end{equation}
					\begin{equation}\label{Trsequence} 
						\begin{tikzcd}
							1\rar& \Rlim_n U_r(R^{p^n})\rar& \Rlim_n
							B_r(R^{p^n})\rar& \Rlim_n
							\G_m^{r}(R^{p^n})\rar\lar[bend
							right=20,swap]{s_R}& 1
						\end{tikzcd}\end{equation}
						with $\Rlim_n U_r(R^{p^n})\rightarrow \Rlim_n B_r(R^{p^n})$
						injective. The canonical section of the projection $B_r\rightarrow\G_m^r $ induces a
						section $s_R$ of the sequence
						\eqref{Trsequence}.
			\item\label{stratAndRlim4} The splitting $s_R$ identifies $\Rlim
				\G^r_m(R^{p^n})$ with the set of isomorphism
				classes of composition series
				\eqref{eq:compSeries} such that $E\cong
				\bigoplus_{i=1}^rE_{i+1}/E_{i}$.
			\item The above identifications and the section $s_R$
				are functorial with respect to
				the inclusions $k[t^{\pm 1}]\subset k\llparen
				t\rrparen$ and $k[t^{\pm 1}]\subset k\llparen
				t^{-1}\rrparen$.
				\end{enumerate}
\end{proposition}
\begin{proof}
	\begin{enumerate}[label={\alph*)}]
		\item\label{item:RlimStrat1}A stratified bundle $E$ on $R$ can be described as a
			sequence of pairs
			\[( E^{(n)},\sigma_n)_{n\geq 0}\]
			where $E^{(n)}$ is a locally free finite rank module on $R^{p^n}$ and
			$\sigma_n$ an $R^{p^n}$-linear isomorphism
			\[\sigma_n:R^{p^{n}}\otimes_{R^{p^{n+1}}}
				E^{(n+1)}\xrightarrow{\cong} E^{(n)},\]
			see \cite[Thm.~1.3]{Gieseker/FlatBundles}. If
			$E':=(E'^{(n)},\sigma'_n)_{n\geq 0}$
			is a second stratified bundle, then a morphism
			of stratified bundles $E\rightarrow E'$ is a
			sequence of $R^{p^n}$-linear morphisms
			$\psi_{n}:E^{(n)}\rightarrow E'^{(n)}$, such that the
			diagrams
			\begin{equation}\label{eq:morphismOfStratifiedBundles}
			\begin{tikzcd}[column sep=.8cm]
				R^{p^n}\otimes_{R^{p^{n+1}}}
				E^{(n+1)}\rar{\sigma_n}\dar[swap]{1\otimes\psi_{n+1}}&
				E^{(n)}\dar{\psi_n}\\
				R^{p^n}\otimes_{R^{p^{n+1}}}
				E'^{(n+1)}\rar{\sigma'_n}& E'^{(n)}
			\end{tikzcd}
		\end{equation}
			commute. Since in our case all vector bundles on $R$
			are free, we may choose a basis for each $E^{(n)}$, so the
			$\sigma_n$ give rise to an element $(\sigma_n)\in
			\prod_{n\geq 0}\GL_r(R^{p^n})$, if $r=\rank E^{(n)}$. If
			$E'$ is a second stratified bundle, and
			$\psi=(\psi_n):E\rightarrow E'$ an isomorphism, then
			we see from \eqref{eq:morphismOfStratifiedBundles} that
			$\sigma_n=\psi_n^{-1}\sigma'_n\psi_{n+1}$. This shows
			that $[(\sigma_n)]=[(\sigma'_n)]\in \Rlim_n
			\GL_r(R^{p^n})$, so the class $[(\sigma_n)]$ is
			independent of the choices and $\Rlim_n \GL_r(R^{p^n})$ is the
			pointed set of isomorphism classes of rank $r$
			stratified bundles on $R$.
		\item Let $E$ and $E'$ be stratified bundles on $R$, with composition
			series
			\[0\subsetneqq E_1\subsetneqq
				E_2\subsetneqq\ldots\subsetneqq
				E_r=E, \quad 0\subsetneqq E'_1\subsetneqq
				E'_2\subsetneqq\ldots\subsetneqq
				E'_r=E'\]
				such that $\rank E_{i+1}/E_i = \rank
				E'_{i+1}/E'_i = 1$ (resp.~such that $E_{i+1}/E_i\cong
				E'_{i+1}/E'_i$ is the trivial stratified
				bundle of rank $1$). Fixing bases of
				$E^{(n)}, E'^{(n)}$,
				adapted to the composition series, the same reasoning as
				in \ref{item:RlimStrat1} shows that an isomorphism of
				the composition series corresponds to a
				sequence of elements $\psi_n\in
				B_r(R^{p^n})$ (resp.~$\psi_n\in
				U_r(R^{p^n})$). The claim follows as before.
			\item Since $\varprojlim_n B_r(R^{p^n})=B_r(k)$ and
				$\varprojlim_n \G_m(R^{p^n})=\G_m(k)$, we
				obtain the short exact sequence
				\eqref{Trsequence}. The map $\Rlim_n
				U_r(R^{p^n})\rightarrow \Rlim_n
				B_r(R^{p^n})$ is injective by Lemma
				\ref{lemma:Rlim}, \ref{item:Rlim2}. The
				section $s_R$ is induced by the functoriality
				of $\Rlim$.
			\item This is clear: The composition series
				corresponding to the image of $s_R$ are split. 
			\item This is also clear due to the functoriality of
				$\Rlim$.
	
	\end{enumerate}
\end{proof}

Next, we briefly recall the notion of regular singularity and the classification of stratified line bundles on $\G_m$ and
$k\llparen t\rrparen$ based on results of \cite{Matzat}.
\begin{definition}\label{defn:oalpha}
	\begin{enumerate}[label={\alph*)}]
		\item\label{defn:oalpha:a} Let $E$ be a stratified bundle on $\Spec k\llparen
			t\rrparen$. Following \cite[Sec.~3]{Gieseker/FlatBundles}, we
			say that $E$ is \emph{regular singular} if there
			exists a free $k\llbracket t\rrbracket$-submodule
			$\mathcal{E}$ of $E$, such that $\dim E=\rank
			\mathcal{E}$ and such that $\mathcal{E}$ is stable
			under the action of the differential operators
			$\delta_t^{(n)}$, $n\geq 0$,  given by
			$\delta^{(n)}_t(t^r)=\binom{r}{n}t^r$.
		\item Let $E$ be a stratified bundle on
			$\G_{m}=\P_k^1\setminus \{0,\infty\}$.
			Fix a coordinate $t$ such that
			$\G_m=\Spec(k[t^{\pm 1}])\subset \Spec
			k[t]=\P^1_k\setminus \{\infty\}$. Then $E$ is said to be
			\emph{regular singular at $0$} if there exists a free
			$k[t]$-submodule $\mathcal{E}\subset E$ such that
			$\rank \mathcal{E}=\rank E$ and such that
			$\mathcal{E}$ is stable under the action of the
			operators $\delta^{(n)}_{t}$, $n\geq 0$. 

			This is easily seen to be equivalent to saying that
			$E\otimes_{\mathcal{O}_{\G_m}}k\llparen t\rrparen$ is
			regular singular in the sense of \ref{defn:oalpha:a},
			where we identify $k\llparen t\rrparen\cong
			\Frac\widehat{\mathcal{O}_{\G_m,0}}$.

			The notion of \emph{regular singularity at $\infty\in
			\P^1_k$} is defined analogously.
			
			We say that $E$ is \emph{regular singular}, if it is regular
			singular both at $0$ and $\infty$.
		\item If $R=k[t^{\pm 1}]$ or $k\llparen t \rrparen$, and $\alpha\in \Z_p$,
			then define $\mathcal{O}_R(\alpha)$ to be the stratification on the
			free $R$-module of rank $1$ given by
			$\delta_t^{(n)}(1)=\binom{\alpha}{n}$. Here $\delta^{(n)}_t$ is
			the differential operator acting via
			\[\delta^{(n)}_t(t^r)=\binom{r}{n}t^r.\]
	\end{enumerate}
\end{definition}
The following proposition is essentially contained in \cite[4.2]{Matzat}.
\begin{proposition}\label{prop:rank1}
	As before let $k$ be an algebraically closed field, $R=k[t^{\pm 1}]$
	or $k\llparen t\rrparen$.
	\begin{enumerate}
		\item\label{item:triviality} If $\alpha\in \Z_p$, then $\mathcal{O}_R(\alpha)$ is
			isomorphic to the trivial stratified bundle if and only
			if $\alpha\in \Z$.
		\item\label{item:homs} For $\alpha,\beta\in \Z_p$ we have
			\begin{align*}
				\hom_{\Strat(R)}(\mathcal{O}_R(\alpha),\mathcal{O}_R(\beta))&=\{0\}\cup\Isom_{\Strat(R)}(\mathcal{O}_R(\alpha),\mathcal{O}_R(\beta))=\\&
				=\begin{cases}k&\text{
				if }\alpha-\beta\in \Z\\0
					&\text{ else.}\end{cases} \end{align*}
				\item $\mathcal{O}_{k\llparen
					t\rrparen}(\alpha)$ is
					regular singular and $\mathcal{O}_{k[t^{\pm 1}]}(\alpha)$  is regular
			singular at $0$ and $\infty$.
		\item\label{item:representativeOfOalpha} The pointed set of isomorphism classes of stratified bundles
			of rank $1$ is an abelian group with
			multiplication given by the tensor product. It is
			isomorphic to 
			the abelian group
			\[\Rlim_n \G_m(R^{p^n})\cong \Z_p/\Z.\]

			The class of $\mathcal{O}_R(\alpha)$ is represented by
			$(t^{\alpha_np^n})_n\in \prod_n\G_m(R^{p^n})$, if
			$\alpha=\sum_{n\geq 0}\alpha_np^n$, $\alpha_n\in
			[0,p)$. 
				
			In particular, every rank $1$ stratified
			bundle on $R$ is isomorphic to
			$\mathcal{O}_R(\alpha)$ for some $\alpha$.
		\item\label{item:picstrat} The inclusion $k[t^{\pm 1}]\hookrightarrow k\llparen
			t\rrparen$ induces a commutative diagram
			\[\begin{tikzcd}
					\Rlim_n \G_m(k[t^{\pm
					1}]^{p^n})\dar[-, double equal sign distance ]\rar&\Rlim_n
					\G_m(k\llparen
					t\rrparen^{p^n})\dar[-,double equal sign distance ]\\
					\Z_p/\Z\rar{\id}&\Z_p/\Z.
				\end{tikzcd}
				\]
		 The
			inclusion $k[t^{\pm 1}]\hookrightarrow k\llparen
			t^{-1}\rrparen$ induces $-\id$ on $\Z_p/\Z$.
	       % \item\label{item:picstrat}  For every stratified bundle of rank $1$ on $R$, there
	       % 	exists an $\alpha\in \Z_p$, such that $L\cong
	       % 	\mathcal{O}_R(\alpha)$.
		\item\label{item:rank1restriction} The inclusion $k[t^{\pm
			1}]\hookrightarrow k\llparen t
			\rrparen$, induces an equivalence on the full
			subcategories of rank $1$ stratified bundles
			\[\Strat^{\rank 1}(\G_m)\xrightarrow{\cong} \Strat^{\rank
			1}(k\llparen t\rrparen).\]
	\end{enumerate}
\end{proposition}
\begin{proof}
	\begin{enumerate}[label={\alph*)}]
		\item Let $\alpha\in \Z_p$ and consider
			$\mathcal{O}_R(\alpha)$. If $\alpha\in \Z$, then it is not difficult to check that
			$\delta^{(n)}_t(t^{-\alpha}\cdot 1)=0$ for every $n\geq
			1$. It follows that $\mathcal{O}_R(\alpha)$ contains a
			trivial subobject, and hence is trivial. Conversely, if
			$\mathcal{O}_R(\alpha)\cong \mathbf{1}_R$, then
			$\mathcal{O}_R(\alpha)$ contains a horizontal element,
			i.e.~there exists $f\in R$, such that
			\[0=\delta^{(n)}_t(f\cdot 1)=\sum_{\begin{subarray}{c}a+b=n\\a,b\geq
					0\end{subarray}}\delta^{(a)}_t(f)\binom{\alpha}{b}\]
			If $m$ denotes the pole
			order of $f$, then it follows that for every $n\geq 0$
			\[\binom{\alpha-m}{n}=\sum_{\begin{subarray}{c}a+b=n\\a,b\geq
					0\end{subarray}}\binom{-m}{a}\binom{\alpha}{b}=0\]
					which means that $\alpha=m$.
		\item Clearly, any nonzero morphism between rank $1$ stratified
			bundles must be an isomorphism, as every such morphism
			is locally split. Since
		\[\Hom_{\Strat(R)}(\mathcal{O}_R(\alpha),\mathcal{O}_R(\beta))=
			\Hom_{\Strat(R)}(\mathcal{O}_R,\mathcal{O}_R(\beta-\alpha))\]
			the claim follows from part \ref{item:triviality}.
		\item This is clear by definition: If $1\in
			\mathcal{O}_{k[t^{\pm 1}]}(\alpha)$ is a basis element such that
			$\delta_t^{(n)}(1)=\binom{\alpha}{n}$, then $1\cdot
			k\llbracket t \rrbracket$ is a $\delta_t^{(n)}$-stable
			lattice at $0$, and $1\cdot k\llbracket
			t^{-1}\rrbracket$ is a $\delta_{t^{-1}}^{(n)}$-stable lattice at $\infty$.  More generally, see
			\cite[Lemma 3.12]{Gieseker/FlatBundles}.
		\item By Proposition \ref{prop:stratAndRlim}, the pointed set
			of stratified bundles of rank $1$ can be identified
			with $\Rlim_n \G_m(R^{p^n})$. 			
			Since $\G_m$ is abelian, this pointed set carries the
			 structure of an abelian
			group.
			
			For $\alpha:=\sum_{n\geq 0}\alpha_n p^n$,
			$\alpha_n\in [0,p)\cap \Z$,
			the class of the rank $1$ stratified bundle $L:=\mathcal{O}_R(\alpha)$ in $\Rlim_n
			\G_m(R^{p^n})$ is the class $[(t^{\alpha_np^n})]$.
			Indeed, if we identify $L$ with a sequence
			$(L^{(n)},\sigma_n)$ as in the proof of Proposition
			\ref{prop:stratAndRlim}, then $L^{(n)}$ is given by 
			\[L^{(n)}=\{x\in L | \delta_t^{(p^m)}(x)=0, 0\leq m<n\},\]
			so if $e$ is a basis of $L^{(0)}$, then
			$t^{-\alpha_0-\alpha_1p-\ldots-\alpha_{n}p^{n}}e$
			is a basis for $L_{n+1}$, and the isomorphism
			$\sigma_n:R^{p^{n}}\otimes_{R^{p^{n+1}}}L^{(n+1)}\xrightarrow{\cong}
			L^{(n)}$ is
			multiplication by $t^{\alpha_np^{n}}$.

			On the other hand, every $f\in (R^{p^n})^\times$ can
			be uniquely written as
			as 
			\[t^{a_np^n}\lambda_n u_n,\] 
			where $a_n\in \Z$, $\lambda_n\in k^\times$ and
			$u_n\in\left(1+t^{p^n}k\llbracket t^{p^n}\rrbracket\right)$
			if $R=k\llparen t\rrparen$, or $u_n=1$ if
			$R=k[t^{\pm 1}]$. This induces a surjective homomorphism
			$\G_m(R^{p^n})/k^\times\rightarrow p^n\Z$, $f\mapsto p^na_n$,
			which is injective if $R=k[t^{\pm 1}]$, and has kernel $1+tk\llbracket t\rrbracket$ if
			$R=k\llparen t\rrparen$.
			We get a morphism of short exact sequences of
			projective systems
			\[
				\begin{tikzcd}
					1\rar&\{\G_m(R^{p^n})/k^\times\}_n\dar[two
					heads]\rar&\G_m(R)/k^\times\dar[two
					heads]\rar&\left\{\G_m(R)/\G_m(R^{p^n})\right\}_{n}\dar[two
					heads]\rar&1\\
					0\rar&\{p^n\Z\}_{n}\rar&\Z\rar&\{\Z/p^n\}_{n}\rar&0
				\end{tikzcd}
				\]
			where the middle terms are constant projective
			systems. Accordingly, we get a morphism of short exact
			sequences
			\[\begin{tikzcd}
					1\rar&\G_m(R)/k^\times\rar\dar[two
					heads]{f}&\varprojlim_n
					\G_m(R)/\G_m(R^{p^n})\rar\dar{\hat{f}}&\Rlim_n
					\G_m(R^{p^n})/k^\times\rar\dar &1\\
					0\rar&\Z \rar& \Z_p\rar& \Z_p/\Z\rar &0
				\end{tikzcd}
				\]
			It is not difficult to see that
			$\ker(f)=\ker(\hat{f})$: If $R=k[t^{\pm 1}]$, this is
			clear; if $R=k\llparen t\rrparen$, then the
			multiplicative group $\ker(f)=1+tk\llbracket
			t\rrbracket$ is already  $p$-adically complete. Moreover, the
			projective system $(1+tk\llbracket
			t\rrbracket)/(1+t^{p^n}k\llbracket t^{p^n}\rrbracket)$
			has surjective transition morphisms, so $\hat{f}$ is
			surjective. We conclude that the induced map \[\Rlim_n
			\G_m(R^{p^n})\cong \Rlim_n
			\G_m(R^{p^n})/k^\times \rightarrow \Z_p/\Z\] is an isomorphism of
			abelian groups. In particular, every stratified line
			bundle is isomorphic to $\mathcal{O}_R(\alpha)$ for
			some $\alpha\in \Z_p$.
			
			Since
			$\mathcal{O}_R(\alpha)\otimes
			\mathcal{O}_R(\beta)\cong
			\mathcal{O}_R(\alpha+\beta)$, we see that the group
			structure of $\Rlim_n \G_m(R^{p^n})$ agrees with the
			group structure defined by the tensor product on the set of isomorphism classes of
			stratified line bundles.
       % 	\item Consider the long exact sequence obtained by applying
       % 		$\varprojlim_n$ to the short exact sequence of projective systems of
       % 		abelian groups
       % 		\[\begin{tikzcd}
       % 				1\rar&
       % 				\{\G_m(R^{p^n})\}_{n}\rar&\{R^\times\}\rar&\left\{R^\times/\left(R^{\times}\right)^{p^n}\right\}_n\rar&1,
       % 			\end{tikzcd}
       % 			\]
       % 		where the middle term is the constant projective system.
       % 		To compute the given representative of
       % 		$\mathcal{O}_R(\alpha)$, we just have to understand the
       % 		map $\varprojlim R^\times/(R^\times)^{p^n}\rightarrow
       % 		\Rlim (R^{\times})^{p^n}$: A sequence $(f_n)\in
       % 		\varprojlim R^\times/(R^\times)^{p^n}$ is mapped to the
       % 		class of $(f^{-1}_nf_{n+1})\in \prod_n
       % 		(R^{\times})^{p^n}$.
		\item This was already contained in the proof of
			\ref{item:representativeOfOalpha}.
		\item This follows from \ref{item:picstrat} and
			\ref{item:homs}.
       % 	\item Every rank $1$ stratified bundle is regular singular by
       % 		\cite[Lemma 3.12]{Gieseker/FlatBundles}. If
       % 		$R=k\llparen t\rrparen$, the claim follows from
       % 		\cite[Thm.~3.3]{Gieseker/FlatBundles}; if
       % 		$R=k[t^{\pm 1}]$, it follows from
       % 		\cite[Thm.~4.3]{Gieseker/FlatBundles}.
	\end{enumerate}
\end{proof}

\section{Special stratified bundles}\label{sec:specialStrat}
As before, $k$ is an algebraically closed field of characteristic $p>0$, and $\G_m:=\P^1_k\setminus\{0,\infty\}$.
\begin{definition}\label{defn:triangle}
	We denote by  $\Strat^{\triangle,
	\rsi}(\G_m)$ (resp.~$\Strat^{\triangle,\rsn}(\G_m)$) 
	the full subcategory of $\Strat(\G_m)$ with objects the
	stratified bundles $E$ which are regular singular at $\infty$
	(resp.~$0$), and which admit a filtration $(E_i)_{i\geq 0}$,
	such that $E_{i+1}/E_i$ has rank $1$. Note that $E_{i+1}/E_i\cong
	\mathcal{O}_{\G_m}(\alpha)$ for some $\alpha\in \Z_p$ 
	according to Proposition \ref{prop:rank1}. In particular,
	$E_{i+1}/E_i$ is regular singular. For the notion of regular
	singularity, see Definition \ref{defn:oalpha}, or, in a more general
	context, \cite[\S3.]{Gieseker/FlatBundles} and
	\cite{Kindler/FiniteBundles}.

	Similarly, write $\Strat^{\unip, \rsi}(\G_m)$
	(resp.~$\Strat^{\unip, \rsn}(\G_m))$ for the full
	subcategory of $\Strat(\G_m)$ with objects unipotent stratified
	bundles which are regular singular at $\infty$ (resp.~$0$).  Here a
	stratified bundle $E$ is called unipotent if it admits a filtration
	$(E_i)_{i\geq 0}$, such that $E_{i+1}/E_i\cong \mathcal{O}_{\G_m}$ as
	stratified bundles.
\end{definition}
The category $\Strat^{\triangle,\rsn}(\G_m)$ is a direct generalization of the category of
\emph{special coverings} of \cite[1.3]{Katz/LocalToGlobal}: Let $f:Y\rightarrow \G_{m}$ be a
finite \'etale covering, and for simplicity assume that $f$ is Galois
with Galois group $G$. The covering $f$ is called \emph{special},
if it is tame at $0$, and if the finite group $G$ contains a
unique (hence normal) $p$-Sylow subgroup $P$. The quotient $G/P$ corresponds
to a tame Galois covering of $\G_m$ and hence is cyclic of order prime to $p$. If $f$ is not Galois, then by definition
$f$ is special if it is the disjoint union of connected coverings, whose
Galois closures are special. The push-forward $f_*\mathcal{O}_Y$ is 
naturally a stratified bundle, and if $f$ is connected, its monodromy group is the finite
constant $k$-group scheme associated with the Galois group of its Galois
closure (\cite[Ch.~VI, 1.2.4.1]{Saavedra}).  
\begin{proposition} \label{prop:special}In the notations from the previous
	paragraph, the covering $f$ is special if and only if
	$f_*\mathcal{O}_Y\in \Strat^{\spc,\rsn}(\G_m)$.
\end{proposition}
\begin{proof}
	Without loss of generality, we may assume $f$ to be Galois.
By \cite[Thm.~6.1]{Kindler/FiniteBundles}, $f$ is tame at $0$ if
	and only if the stratified bundle $f_*\mathcal{O}_Y$ is
	regular singular at $0$. The condition on $G$ is equivalent to
	the condition that every irreducible representation of $G$ on
	finite dimensional $k$-vector spaces has rank $1$. Indeed, if
	$f$ is special and $V\neq 0$ a representation of $G$, then
	$V^P\neq 0$, since $P$ is a $p$-group, and $V^P$ is a
	$G$-representation as $P$ is normal in $G$. If $V$ is
	irreducible, then $V^P=V$, so $V$ comes from a representation
	of the cyclic group $G/P$; the irreducible representations of $G/P$
	are all of rank $1$. 
	
	Conversely, if the irreducible
	representations of $G$ all have rank $1$, then $G$ can be
	realized as a closed subgroup of the group of upper triangular
	matrices of some rank. Its unique $p$-Sylow subgroup is the
	subgroup of unipotent matrices. Thus $f$ is special in the
	sense of \cite{Katz/LocalToGlobal} if and only if
	$f_*\mathcal{O}_Y\in \Strat^{\spc, 0}(\G_m)$. 
\end{proof}

\begin{remark}\label{rem:MainThm}
	\begin{enumerate}[label={\alph*)},ref={\alph*)}]
		\item The categories $\Strat^{\triangle,\rsi}(\G_m)$ and
			$\Strat^{\spc,0}(\G_m)$ are  strictly full
			subtannakian categories of $\Strat(\G_m)$. They are {not} stable under taking extensions, as
			the property of being regular singular at $\infty$
			(resp.~$0$) is not stable under taking extensions.
			For example, let $f:\G_m\rightarrow \G_m$ be the
			Artin-Schreier covering given by $u^p-u=t^{-1}$. The
			rank $2$ subbundle of $f_*\mathcal{O}_{\G_m}$ spanned
			by $1$ and $u$
			is a stratified subbundle which is regular singular at
			$\infty$ and not regular singular at $0$. 
			It is easily seen to be an
			extension of two trivial stratified bundles of rank $1$.	
		\item A stratified bundle $E$ on $\G_m$ which lies in both
			$\Strat^{\unip, \rsn}(\G_m)$ and
			$\Strat^{\unip,\rsi}(\G_m)$ is trivial. Indeed, such an
			$E$ is regular singular with exponents $0\in
		\Z_p/\Z$, so by \cite[Cor.~5.4]{Kindler/FiniteBundles} it extends to a stratified bundle on
		$\P^1_k$  and hence is trivial
		(\cite[Thm.~2.2]{Gieseker/FlatBundles}).
%
       % 		More generally, if $E$ lies in both $\Strat^{\triangle,
       % 		\rsn}(\G_m)$ and $\Strat^{\triangle,\rsi}(\G_m)$, then
       % 		$E$ is a direct sum of
       % 		stratified bundles of rank $1$; this is the content of
       % 		\cite[Prop.~4.3]{Gieseker/FlatBundles}. 	
	\item\label{rem:MainThm:c} The condition that a special stratified bundle admits a
			filtration with graded pieces of rank $1$ is forced on
			us,
			as this property holds for all stratified bundles on
			$k\llparen t\rrparen$ by \cite[Prop.~6.3]{Matzat}.

			This is a striking difference to the situation in
			characteristic $0$.
		The category $\Strat(k\llparen t\rrparen)$ is at the same time simpler
			and more complicated than the category $\DM(\C\llparen
			t\rrparen)$ of
			differential modules on $\C\llparen t\rrparen$: On the one hand, it is
			not true that every irreducible differential module on
			$\C\llparen t\rrparen$ is of rank $1$, but on the
			other hand, if $L_1,L_2$ are differential modules of
			rank $1$, then
			$\Ext^1_{\DM(\C\llparen t\rrparen)}(L_1,L_2) = 0$
			unless $L_1\cong L_2$. The nontrivial irreducible
			objects of $\Strat(k\llparen t\rrparen)$ have rank
			$1$, but if $L_1,L_2$ are
			such objects, then the group $\Ext^1_{\Strat(k\llparen
			t\rrparen )}(L_1,L_2)$
			can be nonzero even if $L_1 \not\cong L_2$.
	\end{enumerate}
\end{remark}

\section{Proof of Theorem \ref{thm:equivalence}}\label{sec:proof}
 We
first show that the restriction functor
\begin{equation}\label{eq:restr}
	\restr:\Strat^{\triangle,\rsi}(\G_m)\rightarrow\Strat(k\llparen t
		\rrparen), E\mapsto
		E|_{k\llparen t\rrparen}
	\end{equation}
is fully faithful.

\begin{proposition}\label{prop:fullyfaithfulness}
	The functor
	\eqref{eq:restr}
	is fully faithful.
\end{proposition}
\begin{proof}
	$\restr$ is faithful by \cite[Prop.~1.19]{DeligneMilne}.

	To see that $\restr$ is full, it suffices to show that for every
	stratified bundle $E\in
	\Strat^{\triangle,\rsi}(\G_m)$, every morphism
	$\mathbf{1}_{k\llparen t\rrparen}\rightarrow E|_{k\llparen t\rrparen}$ lifts to a morphism
	$\mathbf{1}_{\G_m}\rightarrow E$. We  induct on the rank of
	$E$. If $E$ has rank $1$, then we can invoke Proposition
	\ref{prop:rank1}, \ref{item:rank1restriction}.
	Assume that we have proved the statement for all bundles of rank $<r=\rank
	E$, and
	fix a morphism $\phi:\mathbf{1}_{k\llparen t\rrparen}\rightarrow
	E|_{k\llparen t\rrparen}$. By assumption, $E$ has
	a subobject $E'$ of rank $1$. If $\phi$ factors through
	$E'|_{k\llparen t\rrparen}$,
	then we are done by induction. If not, then the composition
	$\mathbf{1}_{k\llparen t\rrparen}\xrightarrow{\phi} E|_{k\llparen
	t\rrparen}\rightarrow (E/E')|_{k\llparen t\rrparen}$ is nonzero, and by
	induction it can be lifted to a nonzero morphism $\mathbf{1}_{\G_m}\rightarrow
	E/E'$. Let $e_1,\ldots, e_r$ be a basis of $E$ such that
	$\delta^{(p^n)}_t(e_i)\subset \left<e_i,\ldots, e_r\right>$
	for every $i,n\geq 0$,
	and such that $E'=\left<e_r\right>=\mathcal{O}_{\G_m}(\alpha)$. Then
	\[\phi(1)=g_1e_1+\ldots+g_{r}e_r\in E|_{{k\llparen
	t\rrparen}}\] is horizontal with
	$g_1,\ldots, g_r\in {k\llparen t\rrparen}$, and we have
	shown that $g_1,\ldots, g_{r-1}\in k[t^{\pm 1}]$. It remains to prove
	that $g_r\in k[t^{\pm 1}]$. For $n\geq 0$ we compute
	\begin{equation}\label{eq:fullness}
		0=\delta^{(p^n)}_t\left(\sum_{i=1}^rg_ie_i\right)=\sum_{i=1}^{r-1}\delta^{(p^n)}_t(g_ie_i)+\sum_{\begin{subarray}{c}a+b=p^n\\a,b\geq
		0\end{subarray}}\delta_t^{(a)}(g_r)\binom{\alpha}{b}e_r.
	\end{equation}
       % Computing the coefficient of $e_r$ in \eqref{eq:fullness} we see that
       % \[
       % 	\sum_{\begin{subarray}{c}a+b=p^n\\a,b\geq
       % 		0\end{subarray}}\delta_t^{(a)}(g_r)\binom{\alpha}{b}\in k[t^{\pm 1}]\]
       % 		for every $n\geq 0$, as $E|_K$
       % comes via restriction from $\G_m$, and since we know that $g_i\in
       % k[t^{\pm 1}]$ for $i<r$.
%	Write $g_r=\sum_{i=-N}^{\infty}g_{ri}t^{i}$ with $g_{ir}\in k$.
%	Then from \eqref{eq:fullness} we see that for every $n$ there exists
%	$i(n)$, such that for $j>i(n)$
%	\[0=g_{jr}\sum_{\begin{subarray}{c}a+b=p^n\\a,b\geq
%			0\end{subarray}}\binom{\alpha}{a}\binom{j}{b}=g_{jr}\binom{\alpha+j}{p^n}.\]
%			Hence, if $g_{jr}\neq 0$, then $\alpha-j\in p^n\Z_p$.
%
       %
       % But
       % this means that $g_r\in k[t^{\pm 1}]$, since by assumption $E$ is
       % regular singular at $\infty$ (i.e.~with respect to the inclusion
       % $k[t^{\pm 1}]\subset k\llparen t^{-1}\rrparen$):
	From this we see that  for every $n\geq 0$
	\[\sum_{\begin{subarray}{c}a+b=p^n\\a,b\geq
			0\end{subarray}}\delta_t^{(a)}(g_r)\binom{\alpha}{b}\in
			k[t^{\pm 1}],\]
			because $g_i\in k[t^{\pm 1}]$ for $i<r$, and 
				$\delta^{(a)}_t(e_i)\in E\subset E|_{k\llparen
				t\rrparen}$.
	This means we can define
	\[d(n):=\deg_t\left( \sum_{\begin{subarray}{c}a+b=p^n\\a,b\geq
					0\end{subarray}}\delta_t^{(a)}(g_r)\binom{\alpha}{b}
					\right).\]
	Since $E$ is regular singular at $\infty$, i.e.~with respect to
	$k[t^{\pm 1}]\subset k\llparen t^{-1}\rrparen$, $d(n)$ is bounded from
	above, say by $M\in \N$. Writing $g_r=\sum_{i=-N}^{\infty}g_{ir}t^i$,
	$g_{ir}\in k$, for all $n\geq 0$ we get
	\begin{equation}\label{eq:degunbounded}\sum_{\begin{subarray}{c}a+b=p^n\\a,b\geq
		0\end{subarray}}\delta_t^{(a)}(g_r)\binom{\alpha}{b}=\sum_{i=-N}^{M}g_{ir}\binom{\alpha+i}{p^n}t^{i}.\end{equation}
		In particular, if $i>M$ and $g_{ir}\neq 0$, then
		$\binom{\alpha+i}{p^n}=0$ for all $n\geq 0$, so $\alpha=-i$.
		This shows that $g_{ir}=0$ for $i\gg 0$, so  $g_{r}$ is a Laurent polynomial.
%
%	If $d(n)$ were bounded from above, say by $M$, then
%	\eqref{eq:degunbounded} would hold for all $i>M$ and
%	all $n>0$. In particular, if $g_{ir}\neq 0$, then $\alpha=-i$.
%
%	then there would be no $k[t^{- 1}]$-lattice
%	for the stratified bundle $E\otimes k\llparen t^{-1}\rrparen$ which is
%	stable under $\delta_{t^{-1}}^{(p^n)}$ and which contains the section
%	$g_1e_1+\ldots+g_{r-1}e_{r-1}\in E$.
%
%	then the order of the coefficients of $e_r$ in
%	$\delta^{(n)}_t(e_i)$ would not be bounded above, so $E$ would have a
%	non-regular pole at $\infty$.
\end{proof}

To prove that the restriction functor is essentially surjective, we first
describe the pointed set of isomorphism classes of composition series of
length $r$ in
$\Strat^{\spc,\rsi}(\G_m)$ using $\Rlim$.
\begin{proposition}\label{prop:reductionToUnipotent}
			The inclusion $k[t^{\pm 1}]\hookrightarrow
				k\llparen t^{-1}\rrparen$ 
				induces a commutative diagram
				%\[\xymatrix{
				%	\Rlim_n B_r( k\llparen
				%	t^{-1}\rrparen^{p^n})\ar[r]&\ar@/_.5cm/[l]_{s_{k\llparen
				%		t^{-1}\rrparen}}\Rlim_n
				%	\G_m^r(k\llparen
				%	t^{-1}\rrparen^{p^n})\ar[r]^--{\cong}&(\Z_p/\Z)^r\\
				%	\Rlim_n B_r( k[t^{\pm
				%	1}]^{p^n})\ar[u]^{\varphi}\ar[r]&\ar@/_.5cm/[l]_{s_{k[t^{\pm
				%	1}]}}\Rlim_n
				%	\G_m^r(k[t^{\pm
				%	1}]^{p^n})\ar[r]^--{\cong}\ar[u]&(\Z_p/\Z)^r\ar[u]_{\cong}
				%}\]
				\[\begin{tikzcd}
						\Rlim_n B_r(k\llparen t^{-1}\rrparen^{p^n})\rar&\Rlim_n \G_m^r(k\llparen t^{-1}\rrparen^{p^n})\lar[bend right=20, swap]{s_{k\llparen t^{-1}\rrparen}}\rar{\cong}&(\Z_p/\Z)^r\\
						\Rlim_n B_r( k[t^{\pm 1}]^{p^n})\uar{\varphi}\rar&\Rlim_n \G_m^r(k[t^{\pm 1}]^{p^n})\lar[bend right=20, swap]{s_{k[t^{\pm 1}]}}\rar{\cong}\uar&(\Z_p/\Z)^r\uar{\cong} .
					\end{tikzcd}
						\]
				The set of isomorphism classes of
				composition series of length $r\geq 1$
				in $\Strat^{\spc,\rsi}(\G_m)$ is given by
				\[\varphi^{-1}(\image(s_{k\llparen
					t^{-1}\rrparen})),\]
				and writing
				\[K_r:=\ker\left(\Rlim_n U_r(k[t^{\pm
				1}]^{p^n})\rightarrow \Rlim_n
						U_r(k\llparen
						t^{-1}\rrparen^{p^n})\right)\]
				the short exact sequence
					\[
						\begin{tikzcd}
							1\rar& \Rlim_n
							U_r(k[t^{\pm 1}]^{p^n})\rar& \Rlim_n
							B_r(k[t^{\pm 1}]^{p^n})\rar& \Rlim_n
							\G_m^{r}(k[t^{\pm 1}]^{p^n})\rar\lar[bend
							right=20,swap]{s_{k[t^{\pm
							1}]}}& 1
						\end{tikzcd}\]
				from Proposition \ref{prop:stratAndRlim}
				restricts to a short exact sequence of pointed
				sets
		        	\begin{equation}\label{Trkersequence}
		        		\begin{tikzcd}
		        			1\rar&	K_r\rar& \varphi^{-1}(\image s_{k\llparen t^{-1}\rrparen}) \rar&\Rlim_n \G_m^r(k[t^{\pm 1}]^{p^n})\rar\lar[bend right=20, swap]{s_{k[t^{\pm 1}]}} & 1
		        		\end{tikzcd}
		        	\end{equation}
				with section $s_{k[t^{\pm 1}]}$.

\end{proposition}

\begin{proof}
This follows almost entirely from Propostition
	\ref{prop:stratAndRlim}.  The description of the preimage $\phi^{-1}(\image(s_{k\llparen t^{-1}\rrparen}))$
	follows from the fact that a stratified bundle $E$ on $\G_m$ is
	regular singular at $\infty$ if and only if $E\otimes k\llparen
	t^{-1}\rrparen$ is a direct sum of rank $1$ objects, see
	\cite[Thm.~3.3]{Gieseker/FlatBundles} or \cite[Prop.~6.1]{Matzat},
	which by Proposition \ref{prop:stratAndRlim}, \ref{stratAndRlim4} means that its class
	maps to $\image s_{k\llparen t^{-1}\rrparen}$. For
	the sequence \eqref{Trkersequence}, note that $K_r= \Rlim_n U_r(k[t^{\pm 1}])\cap \varphi^{-1}(\image(s_{k\llparen t^{-1}\rrparen}))$.
\end{proof}

Theorem \ref{thm:equivalence} now follows from the following proposition.
\begin{proposition}\label{prop:twistedMain}With the notation from Proposition
	\ref{prop:reductionToUnipotent}, the inclusion $k[t^{\pm 1}]\hookrightarrow
	k\llparen t\rrparen$ induces a surjection of pointed
	sets
	\begin{equation}\label{phi}\phi:\phi^{-1}(\image s_{k\llparen t^{-1}\rrparen})\twoheadrightarrow
		\Rlim_n B_r(k\llparen t\rrparen^{p^n}).\end{equation}
		In particular, the restriction functor $\Strat^{\triangle,
		\rsi}(\G_m)\rightarrow \Strat(k\llparen t \rrparen)$ is
		essentially surjective.

		Note that \eqref{phi} is a priori injective, according to Proposition
		\ref{prop:fullyfaithfulness}.
\end{proposition}

An obvious approach would be to use \eqref{Trkersequence} together with some
variant of a
``5-lemma'' to reduce to the unipotent case: According to Proposition \ref{prop:stratAndRlim}, the inclusion $k[t^{\pm 1}]\subset k\llparen t\rrparen$
gives rise to a morphism of short exact sequences of pointed sets
%\[\maxsizebox{\textwidth}{!}{\xymatrix@C=.5cm{
%	1\ar[r] &K_r\ar[r]\ar[d]& \ar[d]\varphi^{-1}(\image s_{k\llparen
%			t^{-1}\rrparen}) \ar[r]&\ar@/_.5cm/[l]_{s_{k[t^{\pm
%			1}]}} \Rlim_n \G_m^r(k[t^{\pm
%			1}]^{p^n})\ar[d]^{\cong}\ar[r]& 1\\
%		1\ar[r] &\Rlim_n U_r(k\llparen t\rrparen^{p^n})\ar[r]& \Rlim_n
%		B_r(k\llparen
%		t\rrparen^{p^n})\ar[r]&\ar@/_.5cm/[l]_{s_{k\llparen t
%		\rrparen}} \Rlim_n \G_m^r(k\llparen t \rrparen^{p^n})\ar[r]& 1
%	}}
%	\]
\[\maxsizebox{\textwidth}{!}{
	\begin{tikzcd}
	1\rar &K_r\rar\dar& \varphi^{-1}(\image s_{k\llparen t^{-1}\rrparen})
	\dar\rar& \Rlim_n \G_m^r(k[t^{\pm 1}]^{p^n})\dar{\cong}\rar\lar[bend right=20, swap]{s_{k[t^{\pm 1}]}}& 1\\
	1\rar &\Rlim_n U_r(k\llparen t\rrparen^{p^n})\rar& \Rlim_n B_r(k\llparen t\rrparen^{p^n})\rar&\Rlim_n \G_m^r(k\llparen t \rrparen^{p^n})\rar\lar[bend right=20, swap]{s_{k\llparen t \rrparen}} & 1
	\end{tikzcd}
}
	\]
	compatible with the sections $s_{k[t^{\pm 1}]}$ and $s_{k\llparen
	t\rrparen}$.
	 The surjectivity of the left vertical arrow is the statement
	 corresponding to Proposition \ref{prop:twistedMain} for unipotent bundles.
	 The vertical arrow on the right is a bijection by Proposition \ref{prop:rank1}.
Now if we had some form of 5-lemma applicable to this situation, the proof of
Proposition \ref{prop:twistedMain} would be reduced to the unipotent case. Unfortunately, I do not know
of such a 5-lemma. The problem is that if $c\in \Rlim_n \G_m^r(k[t^{\pm 1}]^{p^n})\cong
(\Z_p/\Z)^{r}$, then a priori the fiber of $\varphi^{-1}(\image s_{k\llparen
t^{-1}\rrparen})$ over $c$ is different from the fiber over $1$, and
since we do not have a uniform structure, there does not seem to be a way to
relate these fibers.

Instead we have to take a more involved path. Fix $c\in
(\Z_p/\Z)^{r}\cong \Rlim_n \G_m^r(k[t^{\pm 1}]^{p^n})$ and let
$\gamma:=(\gamma_n)_n\in \prod_n \G_m^r(k[t^{\pm 1}]^{p^n})$ be a representative of
of $c$. We ``twist'' the projective systems by $\gamma$:
Let $R=k[t^{\pm 1}]$ or $k\llparen t\rrparen$ and let
$\{B_r(R^{p^n})\}^\gamma$ denote the projective system with groups
$B_r(R^{p^n})$ and transition morphisms
\[u_{n,\gamma}:B_r(R^{p^{n+1}})\rightarrow B_r(R^{p^n}),\;
	u_{n,\gamma}(A_{n+1})=s_R(\gamma_n^{-1})A_{n+1}s_R(\gamma_{n}),\] where
$s_R:\G^r_m(R)\rightarrow B_r(R)$ is the section of $B_r(R)\rightarrow
\G_m^r(R)$, 
attaching to an ordered $r$-tuple of units the corresponding diagonal matrix. Write
$\{U_r(R^{p^n})\}_n^{\gamma}$ for the induced projective system. We get
a short exact sequence of projective systems
\[
	\begin{tikzcd}
		1\rar& \{U_r(R^{p^n})\}_n^{\gamma}\rar&
		\{B_r(R^{p^n})\}_n^{\gamma}\rar&
		\{\G_m^r(R^{p^n})\}_n\rar\lar[swap, bend right=20]{s_R}& 1,
	\end{tikzcd}
\]
where the twist on the right hand term is trivial, as
$\G_m^r(R^{p^n})$ is abelian. We write $\varprojlim^\gamma$ and
$\Rlim^{\gamma}$ for the limit and derived limit of the twisted
systems. Note that the section $s_R$ induces a section of
$\{B_r(R^{p^n})\}_n^{\gamma}\twoheadrightarrow \{\G_m^r(R^{p^n})\}_n$, so we obtain a short exact sequence of pointed
sets
\[
	\begin{tikzcd}
		1\rar& \Rlim^\gamma_n U_r(R^{p^n})\rar[hook]&
	\Rlim^\gamma_n B_r(R^{p^n})\rar& \Rlim_n \G_m^r(R^{p^n})\rar& 1
\end{tikzcd}
\]
by Lemma \ref{lemma:Rlim}.
Moreover, there
is a canonical bijection
\[\Rlim_n B_r(R^{p^n})\xrightarrow{\cong}\Rlim_n^{\gamma}
	B_r(R^{p^n})\]
mapping a class $[ (b_n)]$ to $[(b_ns_R(\gamma_n)))]$. It is straight-forward
to check that this is well-defined. We get a commutative diagram
with bijective vertical arrows and exact rows:
%\[
%	\xymatrix@C=.5cm{
%		0\ar[r]&\Rlim_n U_r(R^{p^n})\ar@{^{(}->}[r]&\Rlim_n
%		B_r(R^{p^n})\ar[r]\ar[d]_{\cong}^{1\mapsto s_R(c)}&\Rlim_n
%		\G_m^{r}(R^{p^n})\ar[d]_{\cong}^{1\mapsto c}\ar[r]\ar@/_.5cm/_{s_R}[l]&0\\
%		0\ar[r]&\Rlim^{\gamma}_n
%		U_r(R^{p^n})\ar@{^{(}->}[r]&\Rlim^{\gamma}_n
%		B_r(R^{p^n})\ar[r]&\Rlim_n
%		\G_m^{r}(R^{p^n})\ar[r]\ar@/^.5cm/[l]_{s_R}&0
%	}
%	\]
\[
	\begin{tikzcd}
		1\rar&\Rlim_n U_r(R^{p^n})\rar[hook]&\Rlim_n
		B_r(R^{p^n})\rar\dar{\cong}[swap]{1\mapsto s_R(c)}&\Rlim_n
		\G_m^{r}(R^{p^n})\dar{\cong}[swap]{1\mapsto c}\lar[bend
		right=20, swap]{s_R}\rar&1\\
		1\rar&\Rlim^{\gamma}_n U_r(R^{p^n})\rar[hook]&\Rlim^{\gamma}_n B_r(R^{p^n})\rar&\Rlim_n
		\G_m^{r}(R^{p^n})\rar\lar[bend right=20, swap]{s_R}&1
	\end{tikzcd}
	\]
This way we can identify the fiber of $\Rlim_n
B_r(R^{p^n})\rightarrow \Rlim_n \G_m^r(R^{p^n})$ over
$c^{-1}$ with $\Rlim^{\gamma}_n U_r(R^{p^n})$.

Coming back to the sequence \eqref{Trkersequence}, we see that we
can identify the fiber of
\[\varphi^{-1}(\image s_{k\llparen t^{-1}\rrparen})\rightarrow
	\Rlim_n \G^r_m(k[t^{\pm 1}]^{p^n})\]
over $c^{-1}$ with
\begin{equation}\label{eq:Kr}K_r^\gamma:=\ker\left( \Rlim_n^{\gamma} U_r(k[t^{\pm 1}]^{p^n})\rightarrow
	\Rlim^{\gamma}_n U_r(k\llparen t^{-1}\rrparen^{p^n})\right)\end{equation}
and our objective is to prove the following lemma.
\begin{lemma}\label{lemma:Gis}
The inclusion $k[t^{\pm 1}]\hookrightarrow k\llparen t \rrparen$ induces a
surjection
\begin{equation}\label{eq:twistediso}
	K_r^\gamma\twoheadrightarrow \Rlim^{\gamma}_n U_r(k\llparen
	t\rrparen^{p^n}),
\end{equation}
for all $c$.
\end{lemma} 
If the lemma holds, the discussion preceding it implies that the map \[\varphi^{-1}(\image s_{k\llparen
	t^{-1}\rrparen})\rightarrow \Rlim_n B_r(R^{p^n})\] is surjective. 
\begin{proof}
To show that
\eqref{eq:twistediso} is surjective we induct on the rank $r\geq 2$;
the case $r=2$ is Lemma
\ref{lemma:twistedRank2} below.

Let $r>2$ and assume that we know that \eqref{eq:twistediso} is surjective for all ranks
$<r$. 
For $i=1,\ldots, r-1$ let $G_i\subset U_r$ be the normal closed subgroup of
$U_r$ given by matrices of the form
\[
\left(
\begin{array}{c|c}
	\raisebox{-40pt}{{\mbox{{$\id$}}}}
	&a_1 \\[-35pt]
	&\vdots \\
	&a_{i}\\
	&0\\
	&\vdots\\
	&0\\\hline
	0\;\cdots\; 0 & 1
\end{array}
\right).
\]
and $G_0=0$. Then the groups $G_i$ are also normal and closed in $B_r$.
Note that $G_i\cong \G^i_a$, and
one easily checks that $G_i/G_{i-1}$ is a
\emph{central} subgroup of $U_r/G_{i-1}$.
Now let $R=k[t^{\pm 1}]$ or $R=k\llparen t\rrparen$. Then
\[(U_r/G_i)(R)=U_r(R)/G_i(R)\text{ and } (G_i/G_{i-1})(R)=G_i(R)/G_{i-1}(R).\]
Write $\{G_i(R^{p^n})\}^{\gamma}$ for
the induced projective system. Note that this is well defined,
as conjugation by a diagonal matrix fixes $G_i\subset B_r$,
and also note that we have
\[\{G_i(R^{p^n})\}^{\gamma}/\{G_{i-1}(R^{p^n})\}^{\gamma}=\left\{\left(G_i/G_{i-1}\right)(R^{p^n})\right\}^{\gamma},\]
and
\[\{U_r(R^{p^n})\}^{\gamma}/\{G_i(R^{p^n})\}^{\gamma}=\left\{\left(U_r/G_i\right)(R^{p^n})\right\}^{\gamma},\]
where both right hand sides are meaningful, again because
$G_i\subset B_r$ is fixed by conjugation with diagonal
matrices.

We claim that for every $i<r$, one obtains a short exact sequence
%\begin{equation}%\label{eq:Gisequence}
%	\hspace{-.3cm}\xymatrix@C=.3cm{
%		1\ar[r]& \varprojlim_n^{\gamma}
%		\left(G_i/G_{i-1}\right)(R^{p^n})\ar[r]&
%	\varprojlim^{\gamma}_n
%	\left(U_r/G_{i-1}\right)(R^{p^n})\ar[r]&\varprojlim^{\gamma}_n
%	\left(U_r/G_i\right)(R^{p^n})\ar[r]&
%1}
%\end{equation}
\begin{equation}\label{eq:Gisequence}	
	\maxsizebox{\textwidth}{!}{
	\begin{tikzcd}[column sep=.3cm]
	1\rar& \varprojlim_n^{\gamma} \left(G_i/G_{i-1}\right)(R^{p^n})\rar& \varprojlim^{\gamma}_n \left(U_r/G_{i-1}\right)(R^{p^n})\rar&\varprojlim^{\gamma}_n \left(U_r/G_i\right)(R^{p^n})\rar&1
\end{tikzcd}}
\end{equation}
We only have to check that the third arrow is surjective. For every $i$ the
projections $G_{r-1}/G_{i-1}\rightarrow G_{r-1}/G_i$ and
$U_r/G_{i-1}\rightarrow U_r/G_{r-1}=U_{r-1}$ have a section. It follows
that for every $i$ we have a commutative diagram
%\begin{equation*}
%	\hspace{-.3cm}
%	\xymatrix@C=.3cm{ 
%		1\ar[r]&\ar@{->>}[d]\varprojlim^{\gamma}(G_{r-1}/G_{i-1})(R^{p^n})\ar[r]&\ar[d]\varprojlim^{\gamma}(U_r/G_{i-1})(R^{p^n})\ar[r]&\varprojlim^{\gamma}U_{r-1}(R^{p^n})\ar[r]\ar@{=}[d]&1\\
%		1\ar[r]&\varprojlim^{\gamma}(G_{r-1}/G_{i})(R^{p^n})\ar[r]&\varprojlim^{\gamma}(U_r/G_{i})(R^{p^n})\ar[r]&\varprojlim^{\gamma}U_{r-1}(R^{p^n})\ar[r]&1
%		 }
%\end{equation*}
\begin{equation*}
	\begin{tikzcd}[column sep=.275cm]
		1\rar&\varprojlim_n^{\gamma}(G_{r-1}/G_{i-1})(R^{p^n})\rar\dar[two
		heads]&\varprojlim_n^{\gamma}(U_r/G_{i-1})(R^{p^n})\rar\dar&\varprojlim_n^{\gamma}U_{r-1}(R^{p^n})\rar\dar[-,double equal sign distance ]&1\\
		1\rar&\varprojlim_n^{\gamma}(G_{r-1}/G_{i})(R^{p^n})\rar&\varprojlim_n^{\gamma}(U_r/G_{i})(R^{p^n})\rar&\varprojlim_n^{\gamma}U_{r-1}(R^{p^n})\rar&1.
	\end{tikzcd}
\end{equation*}
We see that the middle vertical arrow is surjective, as claimed.\footnote{
Here we slightly abuse notation by writing
$\varprojlim^{\gamma}U_{r-1}(R^{p^n})$ for the inverse limit of the projective
system $\{U_{r-1}(R^{p^n})\}$ twisted by the \emph{image} of $\gamma$ under
the projection $\prod_n \G^{r}_m(R^{p^n})\rightarrow \prod_n
\G_m^{r-1}(R^{p^n})$ onto the first $r-1$ factors.}

Since
$G_i/G_{i-1}$ is central in $U_r/G_{i-1}$, Lemma \ref{lemma:Rlim},
\ref{item:Rlim3}, shows that
 we obtain from
\eqref{eq:Gisequence} a
strongly exact (Definition \ref{defn:strongexactness})
short exact sequence of pointed sets
%\[\maxsizebox{\textwidth}{!}{
%	\xymatrix@C=.5cm{
%		1\ar[r]& \Rlim_n^{\gamma}
%		\left(G_i/G_{i-1}\right)(R^{p^n})\ar[r]&
%	\Rlim^{\gamma}_n
%	\left(U_r/G_{i-1}\right)(R^{p^n})\ar[r]&\Rlim^{\gamma}_n
%	\left(U_r/G_i\right)(R^{p^n})\ar[r]&
%1}
%}
%	\]
\[\maxsizebox{\textwidth}{!}{
	\begin{tikzcd}
		1\rar& \Rlim_n^{\gamma}
		\left(G_i/G_{i-1}\right)(R^{p^n})\rar&
	\Rlim^{\gamma}_n
	\left(U_r/G_{i-1}\right)(R^{p^n})\rar&\Rlim^{\gamma}_n
	\left(U_r/G_i\right)(R^{p^n})\rar&
	1\end{tikzcd}
}
\]
If we write
\begin{align*}
	J^{\gamma}_{i-1}&:=\ker\left(\Rlim^{\gamma}_n (G_i/G_{i-1})(k[t^{\pm
	1}]^{p^n})\rightarrow \Rlim^{\gamma}_n(G_i/G_{i-1})(k\llparen t^{-1}\rrparen^{p^n})\right)\\
	K^{\gamma}_{r,i-1}&:=\ker\left(\Rlim^{\gamma}_n (U_r/G_{i-1})(k[t^{\pm
	1}]^{p^n})\rightarrow \Rlim^{\gamma}_n(U_r/G_{i-1})(k\llparen
	t^{-1}\rrparen^{p^n})\right),
%J_i&:=\ker\left(\Rlim_n \left(U_r(k[t^{\pm 1}]^{p^n})/G_{i}(k[t^{\pm 1}]^{p^n})\right)\rightarrow \Rlim_n\left(U_r(k\llparen t^{-1}\rrparen^{p^n})/G_{i}(k\llparen t^{-1}\rrparen^{p^n})\right)\right)\\
\end{align*}
then $K^{\gamma}_{r,0}=K^{\gamma}_r$, and
$K^{\gamma}_{r,r-1}=K^{\gamma}_{r-1,0}=K^{\gamma}_{r-1}$ in the notation from
\eqref{eq:Kr}.
Lemma \ref{lemma:snake} shows that we have strongly exact
short exact sequences of pointed sets
\[\begin{tikzcd}
		1\rar& J_{i-1}^{\gamma}\rar& K^{\gamma}_{r,i-1}\rar&
		K^{\gamma}_{r,i}\rar& 1,
\end{tikzcd}\]
together with morphisms
%\[\maxsizebox{\textwidth}{!}{\xymatrix@C=.5cm{
%	1\ar[r]&\ar[r]\ar[d]^{f_{i-1}} J^{\gamma}_{i-1}\ar[r]&
%	\ar[d]^{g_{i-1}}K^{\gamma}_{i-1}\ar[r]&
%	\ar[d]^{g_{i}}K_i^{\gamma}\ar[r]&1\\
%	1\ar[r]&\ar[r] \Rlim_n^{\gamma}(G_i/G_{i-1})(k\llparen t\rrparen^{p^n})\ar[r]& \Rlim^{\gamma}_n
%	\{(U_r/G_{i-1})(k\llparen t \rrparen^{p^n})\ar[r]&
%	\Rlim^{\gamma}_{n}(U_r/G_i)(k\llparen
%	t\rrparen^{p^n})\ar[r]&1
%}}\]
\[\maxsizebox{\textwidth}{!}{
	\begin{tikzcd}[column sep=.3cm]
	1\rar&J^{\gamma}_{i-1}\dar{f_{i-1}} \rar&
	K^{\gamma}_{r,i-1}\dar{g_{i-1}}\rar& K_{r,i}^{\gamma}\dar{g_{i}}\rar&1\\
	1\rar&\Rlim_n^{\gamma}(G_i/G_{i-1})(k\llparen t\rrparen^{p^n})\rar& \Rlim^{\gamma}_n \{(U_r/G_{i-1})(k\llparen t \rrparen^{p^n})\rar& \Rlim^{\gamma}_{n}(U_r/G_i)(k\llparen t\rrparen^{p^n})\rar&1
\end{tikzcd}}\]
Since $G_i/G_{i-1}=\G_a$,  Lemma \ref{lemma:twistedRank2} below shows that $f_{i-1}$ is a
bijection for $i=1,\ldots, r$. By induction hypothesis we know that
\[g_{r-1}:K^{\gamma}_{r,r-1}=K^{\gamma}_{r-1}\rightarrow
	\Rlim_n^{\gamma}U_{r-1}(k\llparen t^{p^n}\rrparen)\] is
surjective. Thus, applying Lemma \ref{lemma:snake} repeatedly, it follows that
$g_{0}$ is  surjective. But $g_0$ is precisely the morphism from
\eqref{eq:twistediso}, which completes the proof, up to verification of Lemma
\ref{lemma:twistedRank2} below.
\end{proof}

\begin{lemma}\label{lemma:twistedRank2}Identify $\G_a$ with $U_2$, and consider
	$\G_m^2$ as the group of invertible diagonal
	rank $2$ matrices.
	Every class $c\in \Rlim_n \G_m^2(k[t^{\pm 1}]^{p^n})$
	has a representative  $\gamma\in \prod_n
	\G_m^{2}(k[t^{\pm 1}]^{p^n})$,
	such that the inclusion $k[t^{\pm
	1}]\hookrightarrow k\llparen t\rrparen$
	induces a bijection
	\[\hspace{-.15cm}\ker\left( \Rlim^{\gamma}_n U_2(k[t^{\pm
	1}]^{p^n})\rightarrow \Rlim^{\gamma}_n U_2(k\llparen
	t^{-1}\rrparen^{p^n})\right)\xrightarrow{\cong}
	\Rlim^{\gamma}_n	U_2(k\llparen t\rrparen^{p^n}).\]
\end{lemma}
\begin{proof}
	If $c$ has a representative
	$\gamma=(\gamma_1,\gamma_2)\in \prod_n
	\G_m(k[t^{\pm 1}]^{p^n})^2$, such that
	$\gamma_1=\gamma_2$, then twisting  the projective
	system $\{U_2(k[t^{\pm 1}]^{p^n})\}_n$ by $\gamma$ has no effect, as
	matrices of the form
	\[\begin{pmatrix}
			\lambda& 0 \\
			0&\lambda\end{pmatrix},\; \lambda \in R\]
	are central in $B_2(R)$. This is a particular case of
	the fact that $\Ext^1(L,L)\cong
	\Ext^1(\mathbf{1},\mathbf{1})$ if $L$ is a rank $1$
	object in a Tannaka category.

	It follows that we may assume that the representative $\gamma$ of $c$
	has the form $\gamma=(1,\gamma_2)$, $\gamma_2=(\gamma_{2,n})_n\in
	\prod_n \G_m(k[t^{\pm 1}]^{p^n})$.
	Moreover by Proposition
	\ref{prop:rank1}, \ref{item:representativeOfOalpha}, we may assume
	that $\gamma_{2,n}\in\G_m(k[t^{\pm 1}]^{p^n})$ has the form
	$\gamma_{2,n}=t^{\alpha_np^n}$, with $\alpha_n\in
	[0,p)$. Write $\alpha:=\sum_{i\geq
	0}\alpha_ip^{i}$, $\alpha_i\in [0,p-1)$, for the corresponding $p$-adic
	integer. If $\alpha\in \Z$, then we may assume
	$\gamma_{2,n}=1$ for all $n$, Proposition \ref{prop:rank1},
	\ref{item:triviality}.

	If we identify $\G_a$ with $U_2$
	via
	\[\G_m(R)\ni a\leftrightarrow \begin{pmatrix}1&a\\0&1\end{pmatrix}\in
			U_2(R),\]
	then the transition morphisms
	$U_2(k\llparen t\rrparen^{p^{n+1}})\rightarrow
	U_2(k\llparen t \rrparen^{p^n})$ in the
	projective system $\{U_2(k\llparen
		t\rrparen^{p^n}\}^{\gamma}$ are $a\mapsto
		t^{\alpha_np^n}a$.

To state the main technical ingredient we need the following definition.
\begin{definition}\label{defn:partition}
	We consider an element $f$ of the additive group $k^{\Z}$ as a formal
	series
	\[f=\sum_{i\in \Z}f_i t^i, f_i\in k.\]
	Define $\supp(f):=\{i\in \Z| f_i\neq 0\}$. If $U\subset \Z$ is a
	subset, then we write $f^{(U)}:=\sum_{i\in U}f_i t^{i}$. If
	$\Z=\bigcup_{j\in \Z} U_j$ with $U_j$ pairwise disjoint subsets, then
	\[ f=\sum_{j\in \Z} f^{(U_j)}.\]

	Now if $R=k[t^{\pm 1}]$, $k\llparen t \rrparen$ or
	$k\llparen t^{-1}\rrparen$, we can consider the additive
	group underlying $R$ as subgroup of $k^\Z$, and the notations above apply
	to elements $f\in R$. In particular, the additive group $k\llparen
	t\rrparen$ (resp.~$k\llparen t^{-1}\rrparen$, resp.~$k[t^{\pm 1}]$) is the subgroup of $k^{\Z}$ consisting of formal series $f$
	with $\supp(f)$ bounded from below (resp.~from above,
	resp.~from above and below).
\end{definition}

\begin{lemma}\label{lemma:triviality}
		Let
		$R=k[t^{\pm 1}]$, $k\llparen t \rrparen$ or
		$k\llparen t^{-1}\rrparen$, and  
		$a\in \Rlim_n^{\gamma}(U_2(R^{p^n}))$, where $\gamma\in
		\prod_{n\geq 0}\G_m^2(R^{p^n})$ is of the form
		$(1,t^{\alpha_np^n})_n$, with $\alpha_n\in [0,p)$ and $\alpha:=\sum_{n\geq
		0}\alpha_np^n$ either $=0$ or $\not\in \Z$. 
		\begin{enumerate}
			\item\label{lemma:nicerep} The class $a$ has a
				representative $(a_n)_n\in
				\prod_{n\geq 0}\G_a(R^{p^n})$, satisfying
				\begin{equation}\label{eq:supportcondition}\supp({a}_n)\subset p^n\Z\setminus
					\left(\alpha_np^n+p^{n+1}\Z\right).\end{equation}
			\item\label{lemma:triv0} If $R=k\llparen t\rrparen$,
				and if $(a_n)_n$ satisfies
				\eqref{eq:supportcondition}, then
				the class of $(a_n)_n$ is
				trivial, if and only if
				$a_n\in k\llbracket
				t\rrbracket$ for $n\gg 1$.
			\item\label{lemma:trivInfty}If $R=k\llparen
				t^{-1}\rrparen$, and if $(a_n)_n$ satisfies
				\eqref{eq:supportcondition}, then
				the class of $(a_n)_n$ is
				trivial, if and only if
				$a_n\in t^{-1}k\llbracket
				t^{-1}\rrbracket$ for $n\gg 1$.
			\item\label{lemma:trivGm} If $R=k[t^{\pm 1 }]$, and if
				$(a_n)_n$ satisfies
				\eqref{eq:supportcondition}, then the
				following are equivalent:
				\begin{enumerate}[label=\emph{{\roman*}.}, ref={\roman*}.]
					\item\label{it:gmtriv1} The class of
						$(a_n)_n$ is trivial.
					\item\label{it:gmtriv2} $(a_n)_n$  is
						trivial in
						$\Rlim_n^{\gamma}U_2(k\llparen
						t\rrparen^{p^n})$ and $\Rlim_n^{\gamma}U_2(k\llparen t^{-1}\rrparen^{p^n})$.
					\item \label{it:gmtriv3}$a_n=0$ for $n\gg1$.
				\end{enumerate}
		\end{enumerate}
	\end{lemma}
	\begin{proof}
		\begin{enumerate}[label={\alph*)},ref={\alph*)}]
			\item An element $f\in R^{p^n}$ can be uniquely
				written as sum
				\[f=f^{(p^n\Z\setminus
					(\alpha_np^n+p^{n+1}\Z))}+f^{(\alpha_np^n+p^{n+1}\Z)};\]
				See Definition
				\ref{defn:partition} for this notation. To increase legibility we
				define
				\[f':=f^{(p^n\Z\setminus(\alpha_np^n+p^{n+1}\Z))},\;
					f'':=t^{-\alpha_np^n}f^{(\alpha_np^n+p^{n+1}\Z)}.\]
					Let $(\tilde{a}_n)_n\in \prod_{n\geq
					1}\G_a(R^{p^n})$ be a representative
					of $a$.
				Define $y_0:=0$, and inductively
				\[y_{n+1}:=\underbrace{y''_n}_{\in
					R^{p^{n+1}}}-\underbrace{\tilde{a}''_n}_{\in
						R^{p^{n+1}}}\in
					R^{p^{n+1}}.\] Then we
				check that 
				\begin{align*}y_n + (\tilde{a}'_n-y'_n) -
					t^{\alpha_np^n}y_{n+1}&=
					(y_n-y'_n)+\tilde{a}'_n-t^{\alpha_np^n}y_{n+1}\\ 
					&=t^{\alpha_np^n}y''_n+\tilde{a}'_n
					-t^{\alpha_np^n}y''_n+t^{\alpha_np^n}\tilde{a}''_n\\
					&=\tilde{a}_n
				\end{align*} 
				so $(\tilde{a}_n)_n$ is equivalent to
				$(a_n)_n:=(\tilde{a}'_n-y'_n)_n$, and
				$\supp(a_n)\subset p^n\Z\setminus
				(\alpha_np^n+p^{n+1}\Z)$.
			\item Let $(a_n)_n$ satisfy
				\eqref{eq:supportcondition}.  First assume that $a_n\in k\llbracket
				t\rrbracket$ for $n\gg 0$. To show that
				the class of $(a_n)_n$ is trivial, it suffices
				to show that the
				sums
				\begin{equation}\label{eq:ysum}
					y_m:=a_m+\sum_{n>m}
					a_nt^{\alpha_mp^{m}+\ldots+\alpha_{n-1}p^{n-1}},
					m\geq 0
				\end{equation}
				exist in $k\llparen t\rrparen$. Indeed, if the
				sums \eqref{eq:ysum} exist, then
				$a_n=y_n-t^{\alpha_np^n}y_{n+1}$, so the
				class of $(a_n)_n$ is trivial. 
				To see that the sums exist, note  that
				\begin{multline*}\supp(a_nt^{\alpha_1p+\ldots+\alpha_{n-1}p^{n-1}})\\\subset 
				\left(\alpha_1p+\ldots+\alpha_{n-1}p^{n-1}+p^n\Z\right)\setminus
					\left(\alpha_1p+\ldots+\alpha_np^{n}+p^{n+1}\Z\right).
				\end{multline*}
				This shows that the sum \eqref{eq:ysum} exists
				in $k^{\Z}$, first with $m=0$, and then for
				all $m\geq 0$.
				Since $a_n\in k\llbracket t\rrbracket$ for
				large $n$, it  follows that the sum
				\eqref{eq:ysum} exists in $k\llparen
				t\rrparen$.

				Conversely,  assume that the class of $(a_n)_n$
				is trivial, i.e.~that there exist $y_n\in
				k\llparen t\rrparen^{p^n}$, such that
				$a_n=y_n-t^{\alpha_np^n}y_{n+1}$.  There
				exists $N\gg 0$, such that for all $n\geq N$,
				\[\supp(y_0)\cap 
					(\alpha_0+\alpha_1p+\ldots+\alpha_{n}p^{n}-\Z_{>0}p^{n+1})=\emptyset,\]
				where 
				$(\alpha_0+\alpha_1p+\ldots+\alpha_np^n-\Z_{>0}p^{n+1})$
				denotes the set of integers $m$, such that
				$\frac{m-(\alpha_0+\ldots+\alpha_np^n)}{p^{n+1}}\in
				\Z_{<0}$.
				Indeed, just take $N$ such that
				$\alpha_0+\alpha_1p+\ldots+\alpha_Np^N - p^{N+1}$ is
				smaller than $\ord(y_0)$. Such an $N$ exists,
				as the sequence
				$\alpha_1p+\ldots+\alpha_np^n-p^{n+1}$ is
				decreasing and unbounded, because by our
				assumptions $\alpha=0$ or
				$\alpha\not\in\Z$ (note that e.g.~for
				$\alpha=-1$ there would be a problem here).
       % 			Indeed, if $\alpha\not\in \Z$, then the
       % 			difference
       % 			\[(\alpha_0+\alpha_1p+\ldots+\alpha_np^n-p^{n+1})-(\alpha_0+\alpha_1p+\ldots+\alpha_{n+1}p^{n+1}-p^{n+2})=p^{n+2}-p^{n+1}(\alpha_{n+1}+1)\]
       % 			is nonzero for infinitely many $n$. 

				Summing up the equations
				$a_n=y_n-t^{\alpha_np^n}y_{n+1}$, we obtain
				\[
					\sum_{i\geq
				0}^na_it^{\alpha_0+\ldots+\alpha_{i-1}p^{i-1}}=y_0-t^{\alpha_0+\ldots+\alpha_np^n}y_{n+1}.
				\]
				By construction 
				\[\supp\left(\sum_{i\geq
				0}^na_it^{\alpha_0+\ldots+\alpha_{i-1}p^{i-1}}\right)
				\cap\left(\alpha_0+\ldots+\alpha_np^n+p^{n+1}\Z\right)=\emptyset.\]
				Hence, if $n> N$, then
				$y_{n}\in k\llbracket t\rrbracket$, and
				also
				$a_{n}=y_{n}-t^{\alpha_{n}p^{n}}y_{n+1}\in
				k\llbracket t\rrbracket$, which is
				what we wanted to show.
			\item The argument is very similar to the one from the
				previous part. As before, assume that
				$(a_n)_n\in \prod_n \G_a(k\llparen
				t^{-1}\rrparen^{p^n})$ satisfies
				\eqref{eq:supportcondition}.
				If the class of $(a_n)_n$ is trivial, then
				there exist $y_n\in k\llparen
				t^{-1}\rrparen^{p^n}$, such that
				\[
					a_n = y_n -
					t^{\alpha_np^n}y_{n+1}.
					\]
				Summing up we see that
				\[\sum_{i=0}^n a_i
					t^{\alpha_0+\ldots+\alpha_{i-1}p^{i-1}}=y_0-t^{\alpha_0+\ldots+\alpha_{n}p^{n}}y_{n+1}
					\]
				Again we have
				\[\supp\left(\sum_{i=0}^n a_i
					t^{\alpha_0+\ldots+\alpha_{i-1}p^{i-1}}\right)\cap
					\left(\alpha_0+\ldots+\alpha_np^{n}+p^{n+1}\Z\right)=\emptyset.\]
				Hence, if $N$ is such that 
				\[\supp(y_0)\cap\left(\alpha_0+\ldots+\alpha_np^n+\Z_{>0}p^{n+1}\right)=\emptyset\] for
				$n\geq N$, then $y_{n}\in
				t^{-1}k\llbracket t^{-1}\rrbracket$ for $n>
				N$, which shows that $a_n\in
				t^{-1}k\llbracket t^{-1}\rrbracket$ for $n>N$.
				Such an $N$ exists by the same reasoning as in
				\ref{lemma:triv0}.

				Conversely, if $a_n\in t^{-1}k\llbracket
				t^{-1}\rrbracket$ for $n\gg 0$, then the sums
				\eqref{eq:ysum} exist by the same argument
				as in \ref{lemma:triv0}, which shows that the
				class of $(a_n)_n$ is
				trivial. We remark that in contrast to the
				case of $k\llparen t\rrparen$, for this argument to
				work, it is important that $a_n$ does not
				have a constant term for $n\gg 0$, as can be
				seen, e.g., by considering the case $\alpha=0$.
			\item If $a_n\in
				k[t^{\pm 1}]$ is such that $a_n=0$ for $n\gg
				1$, then the sums \eqref{eq:ysum} exist, so
				the class of $(a_n)_n$ is trivial in
				$\Rlim^{\gamma}U_2(k[t^{\pm 1}]^{p^n})$.
				Hence we have \ref{it:gmtriv3} $\Rightarrow$
				\ref{it:gmtriv1}. We clearly also have
				\ref{it:gmtriv1} $\Rightarrow$
				\ref{it:gmtriv2}.

				Finally, if the class of $(a_n)_n$ is trivial
				over both $k\llparen t\rrparen$ and $k\llparen
				t^{-1}\rrparen$, then for $n\gg 0$ we know
				from \ref{lemma:triv0} and
				\ref{lemma:trivInfty}
				that \[a_n\in k[t^{\pm 1}]\cap k\llbracket
				t\rrbracket\cap t^{-1}k\llbracket
				t^{-1}\rrbracket=\{0\},\]
				so we are done.
		\end{enumerate}
	\end{proof}
	Returning to the proof of Lemma \ref{lemma:twistedRank2}, let $a\in
	\Rlim_n^{\gamma}\G_a(k\llparen t\rrparen^{p^n})$ be a class represented
	by 
	$(a_n)_n\in \prod_n \G_a(k\llparen t\rrparen^{p^n})$ satisfying the
	support condition \eqref{eq:supportcondition}. Every class $a$ has
	such a representative by Lemma \ref{lemma:triviality}.

	For an element $f\in k\llparen t\rrparen$ (or $k\llparen
	t^{-1}\rrparen$ or $k[t^{\pm 1}]$), we write $f=f^{\geq
	0}+f^{<0}$, where $\supp(f^{\geq 0})\subset \Z_{\geq 0}$ and
	$\supp(f^{<0})\subset \Z_{<0}$.

	We see that $(a_n)_n=(a_n^{\geq 0})_n+(a_n^{<0})_n$, and by Lemma
	\ref{lemma:triviality}, the class of $(a_n^{\geq 0})_n$ in $\Rlim_n
	\G_a(k\llparen t\rrparen^{p^n})$ is trivial.
	 It
	follows that the map 
	\[\Rlim_n^{\gamma}\G_a(k[t^{\pm 1}]^{p^n})\rightarrow
	\Rlim_n^{\gamma}\G_a(k\llparen t\rrparen^{p^n})\]
	induced by the inclusion $k[t^{\pm 1}]\hookrightarrow k\llparen t\rrparen$ is surjective. 
	
	Moreover, since $a_n^{<0}\in t^{-1}k[t^{-1}]$, the class of
	$(a_n^{<0})_n$ in
	$\Rlim^{\gamma}_n \G_a(k[t^{\pm 1}]^{p^n})$ maps to the trivial class
	in $\Rlim^{\gamma}_n \G_a(k\llparen t^{-1}\rrparen^{p^n})$, so  the
	map
	\[\hspace{-.15cm}\ker\left( \Rlim^{\gamma}_n U_2(k[t^{\pm
	1}]^{p^n})\rightarrow \Rlim^{\gamma}_n U_2(k\llparen
	t^{-1}\rrparen^{p^n})\right)\rightarrow
	\Rlim^{\gamma}_n	U_2(k\llparen t\rrparen^{p^n})\]
	is surjective. By Lemma \ref{lemma:triviality}, \ref{lemma:trivGm} it is also injective,
	so the proof is complete.
\end{proof}

\section{Applications}\label{sec:applications}
In this section we write $\Vectf_k$ for the category of finite dimensional
$k$-vector spaces, and if $G$ is an affine $k$-group scheme, we write $\Repf_k
G$ for the category of $k$-linear representations of $G$ on finite dimensional
$k$-vector spaces.

First we prove Theorem \ref{thm:ses}.
\begin{proof}[Proof of Theorem \ref{thm:ses}]
       We first show that every object of
       $\Repf_k P(\omega)$ is a successive extension of trivial
       representations of rank $1$, i.e.~that $P(\omega)$ is unipotent. If
       $i:P(\omega)\hookrightarrow \pi_1^{\Strat}(k\llparen
       t\rrparen, \omega)$ denotes the inclusion, then every object of
       $\Repf_k P(\omega)$ is a subquotient of $i^*E$ for some $E\in
       \Strat(k\llparen t\rrparen)$ by \cite[Prop.~2.21]{DeligneMilne}.
       As noted in Remark \ref{rem:MainThm}, \ref{rem:MainThm:c}, every object
       $E$ of $\Strat(k\llparen t\rrparen)$ is a successive extension
       of rank $1$ objects, so it suffices to prove that $i^*E$ is trivial, if
       $\rank E=1$.
       In this case, according to Proposition \ref{prop:rank1}, $E$ is isomorphic to $\mathcal{O}_{k\llparen t\rrparen}(\alpha)$ for
       some $\alpha\in \Z_p$ and
       hence lies in $\Strat^{\rs}(k\llparen t\rrparen)\subset
       \Strat(k\llparen t\rrparen)$, so $i^*E$ is trivial. We have proved that
       $P(\omega)$ is unipotent.
       %Now let $\rank E=r>1$ and by induction
       %assume that every representation of $P(\omega)$ of rank $<r$ is
       %unipotent. Let $V\subset i^*E$
       %a subobject.  Every stratified
       %bundle on $k\llparen t\rrparen$ contains a subobject of rank $1$; in particular $E$
       %contains a subobject $E_1$ of rank $1$, and $V\cap i^*E_1$ is trivial
       %of rank $\leq 1$. Moreover,
       %\[ V/(i^*E_1\cap
       %        V)\subset i^*(E/E_1)\] 
       %     so by induction it follows that 
       %        $V$ is unipotent. As every quotient of a unipotent representation is
       %unipotent, this shows that every subquotient of $i^*E$ is unipotent,
       %and hence that $P(\omega)$ is unipotent.
       % In the category $\Repf_k P(\omega)$, every rank $1$ object is trivial,
       % as $\Strat^{\rs}(k\llparen t\rrparen)\subset \Strat(k\llparen
       % t\rrparen)$ contains all rank $1$ stratified bundles. In
       % $\Strat(k\llparen t\rrparen)$ an object is irreducible if and only if
       % if it has rank $1$. If  $i:P(\omega)\hookrightarrow
       % \pi^{\Strat}(k\llparen t\rrparen, \omega)$ denotes the inclusion, then
       % every object of $\Repf_k P(\omega)$ is a subquotient of an object of
       % the form $i^*E$. It follows that the irreducible objects of $\Repf_k
       % P(\omega)$ are precisely the rank $1$ objects, which are all trivial.
       % But this means that $P(\omega)$ is unipotent.

	For the reducedness of $P(\omega)$ one shows just like in
	\cite{DosSantos}
	that the relative Frobenius homomorphism 
	\begin{equation}\label{eq:relativeFrobenius}\pi_1^{\Strat}(k\llparen
	t\rrparen)\rightarrow \pi_1^{\Strat}(k\llparen
	t\rrparen)^{(1)}\end{equation} is an isomorphism. Indeed, if $F_{/k}$ denotes the
	relative Frobenius for $k\llparen t\rrparen$, then $F_{/k}^{*}$
	is an equivalence
	\[\Strat(k\llparen t\rrparen)\rightarrow \Strat(k\llparen
		t\rrparen^{(1)}).\]
		As in \emph{loc.~cit.~}this is translated
		into the fact that \eqref{eq:relativeFrobenius} is an
		isomorphism.
		
		The same is true for
		$\pi_1^{\rs}(k\llparen t\rrparen)$, as $F_{/k}^*$ restricts to
		an equivalence
		\[\Strat^{\rs}(k\llparen t\rrparen)\rightarrow
			\Strat^{\rs}(k\llparen
		t\rrparen^{(1)}).\]
	Thus the
	relative Frobenius $P(\omega)\rightarrow P(\omega)^{(1)}$ also is an
	isomorphism which implies that $P(\omega)$ is reduced.
\end{proof}

Next we prove Theorem \ref{thm:A1}.
\begin{proof}[Proof of Theorem \ref{thm:A1}]
	It follows from Theorem \ref{thm:equivalence} that the inclusion
	$k[t^{\pm 1}]\subset k\llparen t^{-1}\rrparen$
	induces an equivalence 
	\[\Strat^{\unip,\rsn}(\G_m)\xrightarrow{\cong}\Strat^{\unip}(k\llparen
		t^{-1}\rrparen).\]
	On the other hand, the inclusion $k[t]\subset k[t^{\pm 1}]$
	also induces an equivalence
	\begin{equation}\label{eq:A1Gm}\Strat^{\unip}(\A^1_k)\xrightarrow{\cong}
		\Strat^{\unip,\rsn}(\G_m).\end{equation}
	Indeed, if $E$ is a stratified bundle on $\G_m$ which is regular
	singular along $0$, then one assigns to it \emph{exponents along
	$0$}, which are elements of
	$\Z_p/\Z$, see \cite[Sec.~3]{Gieseker/FlatBundles},
	\cite{Kindler/FiniteBundles}. In fact, if $E$ is
	regular singular at $0$, then $E\otimes k\llparen
	t\rrparen\cong\bigoplus_{i=1}^{\rank E} \mathcal{O}_{k\llparen
	t\rrparen}(\alpha_i)$, with
	$\alpha_i\in \Z_p$, and the classes of $\alpha_i$ in $\Z_p/\Z$ are the
	exponents of $E$ along $0$. If $E$ is unipotent, $\alpha_i\equiv 0\mod
	\Z$. By \cite[Cor.~5.4]{Kindler/FiniteBundles} this implies that $E$
	extends to a stratified bundle $\overline{E}$ on $\A^1_k$.
	 The stratified bundle $\overline{E}$ is unipotent, as the restriction functor induces an equivalence
	$\left<\overline{E}\right>_{\otimes}\xrightarrow{\cong}
	\left<E\right>_{\otimes}$ by \cite[Lem.~2.5]{Kindler/FiniteBundles}.
	This means that \eqref{eq:A1Gm} is
	essentially surjective, but by \emph{loc.~cit.} it is also fully
	faithful. The proof is complete.
\end{proof}

Before we can prove Theorem \ref{thm:coproduct}, we briefly discuss the
coproduct of two affine $k$-group schemes.

\begin{definition}[\cite{Unver/PurelyIrregular}]
	Let $G_1, G_2$ be two affine $k$-group schemes.	Define the
	category $\mathcal{C}(G_1,G_2)$ as follows:
	Objects are triples $(\rho_1,\rho_2,V)$ with $V$ a
	finite dimensional $k$-vector space, and
	$\rho_i:G_i\rightarrow \GL(V)$, $i=1,2$, representations. A morphism
	$(\rho_1,\rho_2,V)\rightarrow(\rho'_1,\rho'_2,V')$ is a
	morphism of $k$-vector spaces $\phi:V\rightarrow V'$, such that  for
	every $k$-algebra $R$, and every $g_i\in G_i(R)$, the diagram
%	\begin{equation*}
%		\xymatrix{ 
%			V\otimes_k R \ar[d]_{
%			\phi\otimes \id}\ar[r]^--{\rho_i(g_i)}& V\otimes_k
%			R\ar[d]^{\phi\otimes \id}\\
%			V'\otimes_k R \ar[r]^--{\rho'_i(g_i)}& V'\otimes_k R
%		}
%	\end{equation*}
	\begin{equation*}
		\begin{tikzcd}[column sep=1.3cm]
			V\otimes_k R \dar[swap]{\phi\otimes \id}\rar{\rho_i(g_i)}& V\otimes_k R\dar{\phi\otimes \id}\\
			V'\otimes_k R \rar{\rho'_i(g_i)}& V'\otimes_k R
		\end{tikzcd}
	\end{equation*}
	commutes for $i=1,2$.
\end{definition}
The following facts are probably folklore, but the author was not able to find
a reference.
\begin{proposition}\label{prop:coproduct}Let $k$ be an algebraically closed
	field, and $G_1,G_2$, two affine $k$-group schemes.
	\begin{enumerate}
		\item The forgetful functor $\omega:\mathcal{C}(G_1,G_2)\rightarrow
			\Vectf_k$, $\omega(\rho_1,\rho_2,V)=V$, makes
			$\mathcal{C}(G_1,G_2)$ into a neutral tannakian category.
			We denote the corresponding affine $k$-group scheme by
			$G_1\ast^{\alg} G_2$.
		\item The forgetful functors $\tau_i:\mathcal{C}(G_1,G_2)\rightarrow \Repf_k
			G_i$, which are defined on objects by $\tau_i( (\rho_1,\rho_2,V))=\rho_i$, induce
			closed immersions $j_i:G_i\rightarrow G_1\ast^{\alg}
			G_2$, $i=1,2$, which make $G_1\ast^{\alg}G_2$ a
			coproduct in the category of affine $k$-group schemes.
		\item\label{item:generatedBy} Let $G$ be an affine $k$-group scheme and
			$j'_i:G_i\hookrightarrow G$ closed immersions. Assume
			that there exists no closed proper subgroup scheme of $G$,
			containing $j'_1(G_1)$ and $j'_2(G_2)$. Then the unique map
			$\gamma:G_1\ast^{\alg}G_2\rightarrow G$ induced by the $j'_i$ is
			faithfully flat.
		\item If $G_1$ and $G_2$ are reduced, then so is
			$G_1\ast^{\alg} G_2$. 
		       % If every algebraic quotient of $G_1$ and $G_2$ is a
		       % smooth algebraic $k$-group, then the same is true for
		       % $G_1\ast^{\alg} G_2$.
	\end{enumerate}
\end{proposition}
\begin{proof}
	\begin{enumerate}[label={\alph*)},ref={\alph*)}]
		\item This is clear.
		\item The $\tau_i$ are $\otimes$-functors, and if
			$\omega_i:\Repf_k G_i\rightarrow \Vectf_k$ is the
			forgetful functor, then $\omega=\omega_i\tau_i$, so
			$\tau_i$ induces a morphism of group schemes
			$j_i:G_i\rightarrow G_1\ast^{\alg} G_2$. The $\tau_i$
			admit sections: The functor $\sigma_i:\Repf_k G_i\rightarrow
			\mathcal{C}(G_1,G_2)$, assigning to a representation
			$\rho_1:G_1\rightarrow \GL(V)$ the triple
			$(\rho_1,G_2\xrightarrow{\triv} \GL(V),V)$, where
			$\triv$ means the trivial representation of $G_2$ on
			$V$, has the property that $\tau_1\sigma_1$ is the
			identity functor on $\Repf_k G_1$. Analogously we
			define $\sigma_2$. This shows that $\tau_i$ is
			essentially surjective, and in particular that the
			corresponding morphism of group schemes
			$j_i:G_i\rightarrow G_{1}\ast^{\alg}G_2$ is a closed
			immersion for $i=1,2$.

			To see that $G_1\ast^{\alg}G_2$ together with
			$j_1,j_2$ is a coproduct in the category of affine
			$k$-group schemes, let $G$ be an affine $k$-group scheme and
			$j'_i:G_i\rightarrow G$ morphisms. Given a
			representation $\rho:G\rightarrow \GL(V)$, we assign
			to it the triple $(j'_1\rho,j'_2\rho,V)$. Clearly this
			extends to a functor $\Repf_k G\rightarrow
			\mathcal{C}(G_1,G_2)$, such that $\Repf_k G\rightarrow
			\mathcal{C}(G_1,G_2)\xrightarrow{\omega} \Vectf_k$ is
			the forgetful functor. Hence we obtain a morphism of
			group schemes $\gamma:G_1\ast^{\alg}G_2\rightarrow G$, such
			that the diagram
			%\begin{equation*}
       % 			\xymatrix{ 
       % 				G_1\ar[dr]^{j'_1}\ar[d]_{j_1}\\
       % 				G_1\ast^{\alg}G_2\ar[r]& G\\
       % 				G_2\ar[u]^{j_2}\ar[ur]_{j'_2}
       % 				 }
       % 			 \end{equation*}
			\begin{equation*}
				\begin{tikzcd}
					G_1\ar{dr}{j'_1}\dar[swap]{j_1}\\
					G_1\ast^{\alg}G_2\rar& G\\
					G_2\uar{j_2}\ar[swap]{ur}{j'_2}
				\end{tikzcd}
				 \end{equation*}
			 commutes. Moreover, $\gamma$ is unique with this
			 property, as the functor $\Repf_k G \rightarrow
			 \mathcal{C}(G_1,G_2)$ is uniquely determined by
			 $j_1,j_2,j'_1,j'_2$.
		 \item To prove that $\gamma:G_1\ast^{\alg}G_2\rightarrow G$ is
			 faithfully flat, it suffices to show that for
			 every representation $\rho$ of $G$, the functor
			 \[\Repf_k G\xrightarrow{\gamma^*} \mathcal{C}(G_1,G_2)\]
			 restricts to an equivalence
			 \[\left<\rho\right>_{\otimes}\xrightarrow{\cong}
				 \left<\gamma\rho\right>_{\otimes}.\]
			Let $G(\rho)$ be the affine group scheme corresponding
			to the Tannaka category $\left<\rho\right>_{\otimes}$
			and the forgetful functor
			$\left<\rho\right>_{\otimes}\subset \Repf_k
			G\rightarrow \Vectf_k$. We obtain a faithfully flat
			map $G\twoheadrightarrow G(\rho)$. It
			suffices to show that the composition $G_1\ast^{\alg} G_2\rightarrow
			G\rightarrow G(\rho)$ is faithfully flat. If $H$
			denotes the image of $G_1\ast^{\alg}G_2$ in $G(\rho)$,
			then $H\times_{G(\rho)}
			G\rightarrow G$ is a closed affine subgroup scheme
			containing the image of $G_1\ast^{\alg}G_2$, so   by
			assumption $H\times_{G(\rho)} G\rightarrow G$ is an
			isomorphism. By faithfully flat descent $H=G(\rho)$.
			Since $G_1\ast^{\alg}G_2\rightarrow H=G(\rho)$ is
			faithfully flat we are done.
		\item We have to show that every algebraic quotient $G$ of
			$G_1\ast^{\alg}G_2$ is reduced.
			Let $H$ be
			the smallest closed subgroup scheme of $G$ containing
			the images of $G_1$ and $G_2$, which are smooth by
			assumption. Hence $H$ is reduced. But by
			\ref{item:generatedBy}
		we know that $H=G$.
	\end{enumerate}
\end{proof}
\begin{definition}\label{defn:coproduct}
	If $G_1,G_2$ are $k$-group schemes, we write $G_1\ast^{\unip}G_2$ for
	the maximal unipotent quotient of $G_1\ast^{\alg}G_2$. We have
	\[\Repf_{k}\left(G_1\ast^{\unip} G_2\right)=\mathcal{C}^{\unip}(G_1,G_2)\subset
		\mathcal{C}(G_1,G_2),\] the full subcategory of triples
	$(\rho_1,\rho_2,V)$, such that the closed subgroup generated by the
	images of $\rho_1,\rho_2$ is unipotent.
\end{definition}

We will also need the following easy fact.
\begin{lemma}\label{lemma:normalizer}Let $G$ be a smooth unipotent algebraic group over $k$,
	and $H\subsetneqq G$ a proper closed subgroup. Then there
	exists a normal proper subgroup $H'\lhd G$ containing $H$.
\end{lemma}
\begin{proof}
	If $\dim G=0$, then $G$ is the constant $k$-group scheme attached to a
	finite $p$-group because $k$ is algebraically closed, and $H$
	corresponds to an abstract subgroup. In this situation it is not difficult to check
	that $H\subsetneqq G$ is contained in a subgroup
	$H'\subsetneqq G$ of index $p$. But in a finite $p$-group one also has
	$H'\subsetneqq N_G(H')$, so $H'$ is normal in $G$.

	If $\dim G>0$, and if $Z$ is the center of $G$, then $\dim Z>0$
	by \cite[Ch.~17, Ex.~5]{Humphreys/LinearAlgebraicGroups}. Thus $\dim
	G/Z<\dim G$, and $H/(Z\cap H)\subsetneqq G/Z$. By induction there
	exists a proper normal subgroup $\overline{H'}\subsetneqq G/Z$  containing
	$H/(Z\cap H)$. If $H'\subsetneqq G$ is the preimage of $\overline{H'}$, then
	$H\subset H'$ and $H'$ is normal in $G$.
\end{proof}
%\begin{proof}
%	We first show that $H$ is strictly contained in its normalizer $N_G(H)$.
%	If $\dim G=0$, then $G$ is a nilpotent finite group, and it is
%	classical that $H$ is strictly contained in its normalizer
%	$N_G(H)$. 		If $\dim G>0$, then $\dim Z(G)>0$ by \cite[Ch.~17,
%	Ex.~5]{Humphreys/LinearAlgebraicGroups}. Now we can follow the proof of
%	\cite[Prop.~17.4]{Humphreys/LinearAlgebraicGroups}: If $H$ contains the
%	center $Z:=Z(G)$ of $G$, then
%	$H/Z\subsetneqq G/Z$, and $G/Z$ is nilpotent of dimension
%	$<\dim G$. By induction $H/Z\subsetneqq N_{G/Z}(H/Z)$, which
%	implies that $H\subsetneqq N_G(H)$. If $Z\not\subset H$,
%	then $H\subsetneq ZH\subset N_G(H)$.
%
%
%	Let $H'$ be maximal among the proper closed subgroups of $G$
%	containing $H$. While such a subgroup might not exist in a
%	general group scheme, it does in our situation as $G$ is a
%	smooth algebraic group over a field.  Since its normalizer $N_G(H')$ is
%	a closed subgroup, the maximality of $H'$ implies that $N_G(H')=G$, so
%	$H'$ is normal.
%\end{proof}

We are now ready to prove Theorem \ref{thm:coproduct}.
\begin{proof}[Proof of Theorem \ref{thm:coproduct}]
	Let $E$ be a unipotent stratified bundle on $\G_m$, and write $G(E,x)$
	for associated  monodromy group, which is a quotient of
	$\pi_1^{\unip}(\G_m,x)$. Note that $G(E,x)$ is smooth by
	\cite{DosSantos}. Let $H$ denote
	the image of $\pi_1^{\unip}(k\llparen t\rrparen,
	\omega_x)\ast^{\unip}\pi_1^{\unip}(k\llparen t^{-1}\rrparen, \omega_x)$ in
	$G(E,x)$, which is also smooth by Proposition \ref{prop:coproduct}.
	Since $G(E,x)$ is unipotent, if $H\neq
	G(E,x)$, then by Lemma \ref{lemma:normalizer} there exists a proper
	normal subgroup $H'\lhd G(E,x)$
	containing $H$.
	In
	other words: There exists a stratified bundle $E'$ on $\G_m$, with monodromy
	group $G(E,x)/H'\neq \{1\}$, which is
	trivial on $k\llparen t\rrparen$ and $k\llparen t^{-1}\rrparen$. This
	is impossible. Thus $H=G(E,x)$, and it follows from Proposition
	\ref{prop:coproduct}, \ref{item:generatedBy} that the map
	\[ \pi_1^{\unip}(k\llparen
		t\rrparen,\omega_x)\ast^{\unip}\pi_1^{\unip}(k\llparen
		t^{-1}\rrparen,\omega_x)\rightarrow\pi_1^{\unip}(\G_m,\omega_x)
		\]
	is faithfully flat.
	To show that it is an isomorphism, it thus remains to show that the
	functor \[\Strat^{\unip}(\G_m)\rightarrow
	\mathcal{C}^{\unip}(\pi_1^{\unip}(k\llparen
	t\rrparen,\omega_x),\pi_1^{\unip}(k\llparen
	t^{-1}\rrparen,\omega_x)),\] given by $E\mapsto
	(E|_{k\llparen t\rrparen}, E|_{k\llparen t^{-1}\rrparen}, E|_x)$ is
	essentially surjective. But this is a consequence of the following
	lemma.
\end{proof}
\begin{lemma}
	Consider the two
	maps
	\[
	\begin{tikzcd}[column sep=.7cm]
			\Rlim_nU_r(k\llparen t\rrparen^{p^n})&\ar[swap,midway]{l}{\phi^+_r}\Rlim_n
		U_r(k[t^{\pm 1}]^{p^n})\ar{r}{\phi^-_r}& \Rlim_n U_r(k\llparen
        	t^{-1}\rrparen^{p^n})
	\end{tikzcd}
	\]
	induced by the inclusions $k[t^{\pm 1}]\hookrightarrow k\llparen
	t\rrparen$, $k[t^{\pm 1}]\hookrightarrow k\llparen t^{-1} \rrparen$
	(see Convention \ref{conv:inclusion}).
	The map of pointed sets
	\[
	\maxsizebox{\textwidth}{!}{
		\begin{tikzcd}[column sep=.3cm]
			(\phi_r^+,\phi_r^-):\Rlim_n U_r(k[t^{\pm
			1}]^{p^n})\ar{r}& \Rlim_nU_r(k\llparen t\rrparen^{p^n}) \times \Rlim_nU_r(k\llparen t^{-1}\rrparen^{p^n})
		\end{tikzcd}
	}
\]
	is surjective.
\end{lemma}
\begin{proof}
	We first prove the statement for $r=2$ and then induct on $r$.
       % For $r=2$ we know from Lemma \ref{lemma:triviality},
       % \ref{lemma:trivGm} that $(\phi_2^+,\phi_2^-)$ is injective. 

	Let $a\in \Rlim_n \G_a(k\llparen t\rrparen^{p^n})$ and $b\in \Rlim_n
	\G_a(k\llparen t^{-1}\rrparen^{p^n})$ be classes represented by
	$(a_n)_n\in \prod_{n\geq 0}\G_a(k\llparen t\rrparen^{p^n})$ and
	$(b_n)_n\in
	\prod_{n\geq 0} \G_a(k\llparen t^{-1}\rrparen^{p^n})$. By Lemma
	\ref{lemma:triviality} we know that the classes of $(a_n^{\geq 0})_n$ and
	$(b_n^{<0})_n$ are trivial, where $a_n^{\geq 0}:=a_n^{(\Z_{\geq 0})}$
	and $a_n^{<0}:=a_n^{(\Z_{<0})}$ in the notation from Definition
	\ref{defn:partition}, and similarly for $b_n$. Thus we may assume that $a_n\in k[t^{-1}]^{p^n}$,
	and $b_n\in k[t]^{p^n}$ for every $n\geq 1$. Then $a_n+b_n\in
	k[t^{\pm 1}]^{p^n}$, and $\phi^+_2( [(a_n+b_n)_n])=a$,
	$\phi^-_2([(a_n+b_n)_n])=b$, so $(\phi_2^+,\phi_2^-)$ is surjective.

	Now let $r>2$, and assume that the lemma is proved for all ranks $<r$.
	Again let $G_i\lhd U_r$ denote the normal
	subgroups given by matrices of the form
\[
\left(
\begin{array}{c|c}
	\raisebox{-40pt}{{\mbox{{$\id$}}}}
	&a_1 \\[-35pt]
	&\vdots \\
	&a_{i}\\
	&0\\
	&\vdots\\
	&0\\\hline
	0\;\cdots\; 0 & 1
\end{array}
\right).
\]
As in
the proof of Lemma \ref{lemma:Gis} we obtain morphisms of 
strongly exact (Definition \ref{defn:strongexactness})
	short exact sequence of pointed sets
       % \begin{equation*}
       % 	\maxsizebox{\textwidth}{!}{\xymatrix@C=.3cm{ 
       % 		1\ar[r]&\Rlim_n (G_{r-1}/G_{r-2})(k\llparen t\rrparen^{p^n})\ar[r]&\Rlim_n (U_r/G_{r-2})(k\llparen t \rrparen^{p^n})\ar[r]&\Rlim_n U_{r-1}(k\llparen t\rrparen^{p^n})\ar[r]&1\\
       % 		1\ar[r]&\Rlim_n (G_{r-1}/G_{r-2})(k[t^{\pm 1}]^{p^n})\ar[r]\ar[u]\ar[d]&\Rlim_n (U_r/G_{r-2})(k[t^{\pm 1}]^{p^n})\ar[r]\ar[u]^{\phi^+_r}\ar[d]_{\phi^-_r}&\Rlim_n U_{r-1}(k[t^{\pm 1}]^{p^n})\ar[r]\ar[u]^{\phi^+_{r-1}}\ar[d]_{\phi^{-}_{r-1}}&1\\
       % 		1\ar[r]&\Rlim_n (G_{r-1}/G_{r-2})(k\llparen t^{-1}\rrparen^{p^n})\ar[r]&\Rlim_n (U_r/G_{r-2})(k\llparen t^{-1} \rrparen^{p^n})\ar[r]&\Rlim_n U_{r-1}(k\llparen t^{-1}\rrparen^{p^n})\ar[r]&1
       % 	}}
       % 	 \end{equation*}
	\begin{equation*}
		\maxsizebox{\textwidth}{!}{
			\begin{tikzcd}[column sep=.3cm,inner sep=.01cm] 
			1\rar&\Rlim_n (G_{r-1}/G_{r-2})(k\llparen t\rrparen^{p^n})\rar&\Rlim_n (U_r/G_{r-2})(k\llparen t \rrparen^{p^n})\rar&\Rlim_n U_{r-1}(k\llparen t\rrparen^{p^n})\rar&1\\
			1\rar&\Rlim_n (G_{r-1}/G_{r-2})(k[t^{\pm 1}]^{p^n})\rar\uar\dar&\Rlim_n (U_r/G_{r-2})(k[t^{\pm 1}]^{p^n})\rar\uar{\phi^+_r}\dar[swap]{\phi^-_r}&\Rlim_n U_{r-1}(k[t^{\pm 1}]^{p^n})\rar\uar{\phi^+_{r-1}}\dar[swap]{\phi^{-}_{r-1}}&1\\
			1\rar&\Rlim_n (G_{r-1}/G_{r-2})(k\llparen t^{-1}\rrparen^{p^n})\rar&\Rlim_n (U_r/G_{r-2})(k\llparen t^{-1} \rrparen^{p^n})\rar&\Rlim_n U_{r-1}(k\llparen t^{-1}\rrparen^{p^n})\rar&1
		\end{tikzcd}
		}
		 \end{equation*}
%		 \[
%		\maxsizebox{\textwidth}{!}{
%		 \begin{tikzpicture}
%			 \draw[help lines] (-1,0) grid (8,2);
%			 \matrix(a)[matrix of math nodes,
%				 row sep=3em, column sep=2.5em,
%			 text depth=0.25ex, inner sep=0cm, draw]
%			 {
%			 1&\Rlim_n (G_{r-1}/G_{r-2})(k\llparen t\rrparen^{p^n})&\Rlim_n (U_r/G_{r-2})(k\llparen t \rrparen^{p^n})&\Rlim_n U_{r-1}(k\llparen t\rrparen^{p^n})&1\\
%			1&\Rlim_n (G_{r-1}/G_{r-2})(k[t^{\pm 1}]^{p^n})&\Rlim_n (U_r/G_{r-2})(k[t^{\pm 1}]^{p^n}){\phi^+_r}&\Rlim_n U_{r-1}(k[t^{\pm 1}]^{p^n})&1\\
%			%1&\Rlim_n (G_{r-1}/G_{r-2})(k\llparen t^{-1}\rrparen^{p^n})&\Rlim_n (U_r/G_{r-2})(k\llparen t^{-1} \rrparen^{p^n})&\Rlim_n U_{r-1}(k\llparen t^{-1}\rrparen^{p^n})&1
%		};
%			 \path[right hook->](a-1-1) edge (a-1-2);
%		%	 \path[->>](a-1-1) edge (a-2-2);
%			 %\path[dotted,->](a-1-2) edge (a-2-2);
%		%	 \path[left hook->](a-1-3) edge (a-1-2);
%		%	 \path[->>](a-1-3) edge (a-2-2);
%			\draw [brown] (current bounding box.south west) rectangle (current bounding box.north east);
%	\end{tikzpicture}}
%		 \]
	 We sloppily also write $\phi_r^+, \phi_r^-$ for the middle vertical maps in this diagram. 

	 Let $c_+\in \Rlim_n (U_r/G_{r-2})(k\llparen t\rrparen^{p^n})$ and $c_-\in
	 \Rlim_n (U_r/G_{r-2})(k\llparen t^{-1}\rrparen^{p^n})$, and write $\bar{c}_+$
	 (resp.~$\bar{c}_-$) for the image of $c_+$ in $\Rlim_n U_{r-1}(k\llparen
	t\rrparen^{p^n})$ (resp.~for the image of $c_-$ in $\Rlim_nU_{r-1}(k\llparen
	t^{-1}\rrparen^{p^n})$). By induction hypothesis, there exists a
	class
	$\hat{\bar{c}}\in \Rlim_n U_{r-1}(k[t^{\pm 1}]^{p^n})$, such that
	$(\phi_{r-1}^+,\phi_{r-1}^{-})(\hat{\bar{c}})=(\bar{c}_+,\bar{c}_-)$.

	Recall that $G_{r-1}/G_{r-2}\cong \G_a$ and that Lemma
	\ref{lemma:Rlim} shows that the group \[\Rlim_n
		(G_{r-1}/G_{r-2})(k\llparen t\rrparen^{p^n})\] acts on $\Rlim_n
	(U_r/G_{r-2})(k\llparen t\rrparen^{p^n})$, such that the orbits are
	precisely the fibers of the projection to $\Rlim_n U_{r-1}(k\llparen
	t\rrparen^{p^n})$, and similarly for $k[t^{\pm 1}]$ and $k\llparen
	t^{-1}\rrparen$. 

	If $c\in \Rlim_n (U_r/G_{r-2})(k[t^{\pm 1}]^{p^n})$ is any lift of
	$\hat{\bar{c}}$, then $\phi_r^+({c}),c_+$ both map to
	$\bar{c}_+$. Similarly, $\phi_r^{-}(c)$ and $c_-$ both map to
	$\bar{c}_{-}$.
	This means there exist elements 
\[g_+\in \Rlim_n (G_{r-1}/G_{r-2})(k\llparen t\rrparen^{p^n}), g_-\in
	\Rlim_n (G_{r-1}/G_{r-2})(k\llparen t^{-1}\rrparen^{p^n}),\]
	such that $g_{+}+c_+=\phi_r^{+}(c)$ and $g_-+c_-=\phi_r^{-}(c)$.

	Since $G_{r-1}/G_{r-2}\cong U_2$, by induction there exists a 
	class \[g\in \Rlim_n (G_{r-1}/G_{r-2})(k[t^{\pm 1}]),\]
	mapping to $g_-$
	and $g_+$.
	Since $\phi^+_r$ and $\phi_r^-$ are equivariant with respect to the
	action of $\Rlim_n (G_{r-1}/G_{r-2})(k[t^{\pm 1}]^{p^n})$, it follows
	that $\phi^+_r(c-g)=c_+$ and $\phi^-_r(c-g)=c_-$. 
	
	%The unicity of $g$
	%also implies that $c-g$ is the only class with that property. Indeed, if
	%$c'$ is a second such class, then there exists a unique $g'\in \Rlim_n
	%(G_{r-1}/G_{r-2})(k[t^{\pm 1}])$ such that $c-g=c'+g'$.
	%But $c_+=\phi_r^+(c-g)=\phi_r^+(c'+g')=c_++\phi_r^+(g')$, which implies
	%that $\phi_r^+(g')=0$, as the action of
	%$\Rlim_n(G_{r-1}/G_{r-2})(k\llparen t\rrparen^{p^n})$ is free.
	%Similarly, $\phi_r^-(g')=0$, so $g'=0$.

	We have proved that the map
	%\begin{multline*}
	%	(\phi_r^+,\phi_r^-):\Rlim_n (U_r/G_{r-2})(k[t^{\pm 1}])\\
	%	\rightarrow \Rlim_n(U_r/G_{r-2})(k\llparen t\rrparen^{p^n})
	%	\times \Rlim_n(U_r/G_{r-2})(k\llparen t^{-1}\rrparen^{p^n})
	%\end{multline*}
	\[
		\maxsizebox{\textwidth}{!}{
			\begin{tikzcd}[column sep=1cm]
			\Rlim_n (U_r/G_{r-2})(k[t^{\pm
			1}])\rar{(\phi_r^+,\phi_r^-)}&
		\Rlim_n(U_r/G_{r-2})(k\llparen t\rrparen^{p^n})
		\times \Rlim_n(U_r/G_{r-2})(k\llparen t^{-1}\rrparen^{p^n})
	\end{tikzcd}}
		\]
	is surjective.

	Now we apply the above argument inductively to the short exact sequences
	%\[ 1\rightarrow G_{r-i}/G_{r-i-1}\rightarrow U_{r}/G_{r-i-1}\rightarrow U_{r}/U_{r-i}\rightarrow 1.\]
	\[ \begin{tikzcd}1\rar& G_{r-i}/G_{r-i-1}\rar& U_{r}/G_{r-i-1}\rar&
			U_{r}/G_{r-i}\rar& 1.\end{tikzcd}\]
	Eventually we arrive at
	\[\begin{tikzcd}1\rar& G_1\rar& U_r\rar& U_{r}/G_1\rar& 1\end{tikzcd}\]
	which completes the proof.
\end{proof}

Finally, we show that one can recover from Theorem \ref{thm:coproduct} the
result of Katz-Gabber (\cite[bottom of p.~98]{Katz/LocalToGlobal}) stating that the morphism
\[\pi^{\et}_1(k\llparen t\rrparen)^{(p)}\ast_{(p)}\pi_1^{\et}(k\llparen
	t^{-1}\rrparen)^{(p)}\rightarrow\pi_1^{\et}(\G_m)^{(p)}\]
induced by compatible choices of base points, is an isomorphism. Here
$(-)^{(p)}$ denotes the maximal pro-$p$-quotient, and $\ast_{(p)}$ the coproduct
in the category of pro-$p$-groups.  For the construction of the coproduct in the
category of pro-$p$-groups, see \cite[Sec.~9.1]{Ribes/2000}.

\begin{proposition}\label{prop:completion}
	Recall that $k$ is an algebraically closed field of positive
	characteristic $p$.
	Let $G$, $G_1, G_2$ be affine, reduced, unipotent $k$-group schemes.
	\begin{enumerate}
		\item The profinite completion $\widehat{G}$ of $G$ is
			a pro-$p$-group. By \emph{profinite completion} we mean the inverse
	 limit over all finite (hence constant) quotients of $G$.
 \item The profinite completion of $G_1\ast^{\unip}G_2$ is the pro-constant
	 group scheme associated with the coproduct
	 $\widehat{G}_1\ast_{(p)}
	\widehat{G}_2$ in the category of pro-$p$-groups.
	 \end{enumerate}
\end{proposition}
\begin{proof}
	\begin{enumerate}[label={\alph*)},ref={\alph*)}]
		\item This is clear, as finite group of unipotent matrices is a finite $p$-group.
		\item The profinite completion of $G_1\ast^{\unip}G_2$ can be
			described as the affine $k$-group scheme associated
			with the full subcategory
			$\mathcal{C}^{\text{finite}}(G_1,G_2)$ of
			$\mathcal{C}^{\unip}(G_1,G_2)$ (see Definition
			\ref{defn:coproduct}) whose objects have
			a finite monodromy group, i.e.~the full subcategory of
			triples $(\rho_1,\rho_2, V)$, such that the closed
			subgroup $H$ of $\GL(V)$ generated by the images of
			$\rho_1$ and $\rho_2$  is a finite $p$-group. In particular
			$\rho_i:G_i\rightarrow \GL(V)$ factors uniquely through the
			profinite completion of $G_i$, for $i=1,2$. This means
			that we get a fully faithful functor
			\begin{equation}\label{eq:completionFunctor}\mathcal{C}^{\text{finite}}(G_1,G_2)\rightarrow
				\Repf^{\cont}_k
				\left(\widehat{G}_1\ast_{(p)}\widehat{G}_2\right),\end{equation}
			where continuous means with respect to the profinite
			topology and the discrete topology on $k$, i.e.~if $G$
			is a profinite group, then a
			representation $G\rightarrow \GL_k(V)$ is continuous
			if and only if the image of $G$ is finite. The functor
			\eqref{eq:completionFunctor} is also essentially surjective:
			A continuous representation
			$\rho:\widehat{G}_1\ast_{(p)}\widehat{G}_2\rightarrow
			\GL(V)$ gives rise to a pair of 
			representations $\rho_i:G_i\rightarrow \widehat{G}_i\rightarrow
			\GL(V)$, $i=1,2$, such that the images of $\rho_1$ and
			$\rho_2$ generate a finite subgroup of $\GL(V)$, which
			is a finite $p$-group, as
			$\widehat{G}_1\ast_{(p)} \widehat{G}_2$ is a pro-$p$-group.
			Thus the triple $(\rho_1,\rho_2,V)$ is an object of
			$\mathcal{C}^{\text{finite}}(G_1,G_2)$ and the proof
			is complete.
	\end{enumerate}
\end{proof}
By Theorem \ref{thm:A1}, the inclusion $k[t]\subset k\llparen
t^{-1}\rrparen$ induces an isomorphism
\[\pi_1^{\unip}(k\llparen t^{-1}\rrparen,\omega_x)\cong
	\pi_1^{\unip}(\A^1_k,\omega_x).\]
	Applying Proposition \ref{prop:completion} to the isomorphism in  Theorem \ref{thm:coproduct}, we
obtain an isomorphism
\[\pi_1^{\et}(\P^1_k\setminus \{0\},x)^{(p)}\ast_{(p)}\pi_1^{\et}(\A^1_k,x)^{(p)}\xrightarrow{\cong}\pi_1^{\et}(\G_m,x)^{(p)}.\]
Choosing  algebraic closures of $k\llparen t\rrparen$ and $k\llparen
t^{-1}\rrparen$, we get two geometric points $x_0:\Spec \overline{k\llparen
t\rrparen}\rightarrow \G_m$ and $x_{\infty}:\Spec \overline{k\llparen
	t^{-1}\rrparen} \rightarrow \G_m$. Fixing isomorphisms (``chemins'') from
$x$ to $x_0$ and $x_\infty$, we obtain an isomorphism
\begin{equation}\label{eq:katzgabberiso}\Gal(\overline{k\llparen t\rrparen}/k\llparen
	t\rrparen)^{(p)}\ast_{(p)}\Gal(\overline{k\llparen
		t^{-1}\rrparen}/k\llparen
	t^{-1}\rrparen)^{(p)}\xrightarrow{\cong}\pi_1^{\et}(\G_m,x)^{(p)},\end{equation}
which is the isomorphism Katz and Gabber construct in
\cite[p.~98]{Katz/LocalToGlobal}, except that they work with strict
henselizations instead of completions.
\begin{remark}
	The isomorphism \eqref{eq:katzgabberiso} constructed in 
	\cite{Katz/LocalToGlobal} contradicts \cite[Thm.~2.30]{Kedlaya}.
	Indeed, a consequence of \emph{loc.~cit.~}would be that the free
	pro-$p$-group
	$\pi_1(\G_m,x)^{(p)}$ is the direct product 
	\[\pi_1(\A^1,x)^{(p)}\times\pi_1(\P^1_k\setminus \{0\},x)^{(p)}.\] But
	it is not hard to check that such a direct product is not a free
	pro-$p$-group by computing
	\[H^2(\pi_1(\A_k^1,x)^{(p)}\times\pi_1(\P^1_k\setminus
		\{0\},x)^{(p)},\F_p)\neq 0.\]
	For a correction, see \cite{Kedlaya/Erratum}.
\end{remark}

%% file: neutralization.bbl
\providecommand{\bysame}{\leavevmode\hbox to3em{\hrulefill}\thinspace}
\providecommand{\MR}{\relax\ifhmode\unskip\space\fi MR }
\providecommand{\MRhref}[2]{%
  \href{http://www.ams.org/mathscinet-getitem?mr=#1}{#2}
}
\providecommand{\href}[2]{#2}